\documentclass[10pt]{article}
\font\smallit=cmti10

\usepackage{amssymb,latexsym,amsmath,epsfig,amsthm} 

\usepackage{amssymb,amsmath,amsthm,latexsym,booktabs,todonotes,color}
\usepackage{tikz,comment,url}
\usepackage{xcolor}
\usetikzlibrary{arrows.meta,
                decorations.markings}
\usepackage{soul} 
\theoremstyle{definition}
\newtheorem{definition}{Definition}
\newtheorem{notation}[definition]{Notation}

\newtheorem{remark}[definition]{Remark}

\theoremstyle{plain}
\newtheorem{lemma}[definition]{Lemma}
\newtheorem{proposition}{Proposition}
\newtheorem{theorem}[definition]{Theorem}

\newtheorem{example}[definition]{Example}

\usepackage{hyperref}
\hypersetup{
    colorlinks=true,
    linkcolor=blue,
    filecolor=magenta,
    urlcolor=cyan,
    pdftitle={Overleaf Example},
    pdfpagemode=FullScreen,
}

\makeatletter

\renewcommand\section{\@startsection {section}{1}{\z@}
{-30pt \@plus -1ex \@minus -.2ex}
{2.3ex \@plus.2ex}
{\normalfont\normalsize\bfseries\boldmath}}

\renewcommand\subsection{\@startsection{subsection}{2}{\z@}
{-3.25ex\@plus -1ex \@minus -.2ex}
{1.5ex \@plus .2ex}
{\normalfont\normalsize\bfseries\boldmath}}

\renewcommand{\@seccntformat}[1]{\csname the#1\endcsname. }

\makeatother

\begin{document}


\begin{center}
\uppercase{\bf Floridian Solitaire: A New\\ Variant of Bulgarian Solitaire
}
\vskip 20pt
{\bf Aaron Meyerowitz}\\
{\smallit Department of Mathematical Sciences, Florida Atlantic University, Boca Raton, FL 33431, USA}\\
{\tt meyerowi@fau.edu}\\
\vskip 10pt
{\bf Stephen J. Curran}\\
{\smallit Department of Mathematics, University of Pittsburgh at Johnstown, Johnstown, PA 15904, USA}\\
{\tt sjcurran@pitt.edu}\\
\vskip 10pt
{\bf Stephen C. Locke}\\
{\smallit Department of Mathematical Sciences, Florida Atlantic University, Boca Raton, FL 33431, USA}\\
{\tt lockes@fau.edu}\\
\vskip 10pt
{\bf Richard M. Low }\\
{\smallit Department of Mathematics and Statistics, San Jose State University, San Jose, CA 95192, USA}\\
{\tt richard.low@sjsu.edu}\\
\end{center}

\centerline{\bf Abstract}
\noindent

 Bulgarian solitaire is a well-studied, no-choice, no-loss, one-player game involving stacks of cards.
 More formally, it is a self-map on the set of partitions of a fixed integer $n.$
 As a finite dynamical system, its long-term behavior is well understood. Every trajectory ends in a cycle.
 The partitions that are in a cycle are parameterized by binary vectors, and the cycles by binary necklaces.
Call a partition {\bf separated} if distinct part sizes differ by at least two.
The vast majority of partitions belonging to a cycle are not separated.
Motivated by this fact, we consider a variant where the player has choices, but is restricted to separated partitions and, if unable to make a legal move, may lose.
We prove that for $n>73$, there are cycles, and hence winning initial positions.
We analyze the game for small values of $n$ and describe computations which, together with our main result, show that there are  cycles  for $n\in \{2,6,8,11,14,16,18,21\}$ and for $n \ge 23$ but for no other $n.$

\thispagestyle{empty}
\baselineskip=12.875pt
\vskip 30pt

\section{Introduction and notation} \label{S:intro}

Bulgarian solitaire was first introduced to the mathematical
community around 1980 by Konstantin Oskolov of the
Steklov Mathematical Institute in Moscow.
The interested reader may refer to B. Hopkins' article \cite{Hopkins2017} on the history of
Bulgarian solitaire for further information.
Bulgarian solitaire begins with an array of $n>1$ cards placed in one or more stacks.
On each turn, the player removes one card from each stack and forms a new stack using the removed cards. In other words, the positions are the partitions of $n$ and, if there are $p$ parts,  the move is to reduce each of these parts by $1$ and create a new part of size $p$, a $p$-{\bf part}. We will call this the $\Omega$-move, write $\Omega(\lambda)$ for the result of applying it to $\lambda,$ and, for a set $S$ of partitions, let $\Omega(S)=\{\Omega(\lambda): \lambda \in S\}.$ When we say that a part is {\bf reduced}, we always mean reduced in size by $1$.

This is a finite dynamical system. For a given $n,$ the {\bf game graph} is the directed graph with nodes corresponding to the partitions of $n$,  and an arc $\lambda \rightarrow \Omega(\lambda)$ for each partition $\lambda$ of $n$. So, the trajectory from each position ends in a cycle of length $1$ (a fixed point) or more. We call a partition that belongs to a cycle a {\bf recurrent state} or {\bf recurrent partition}.

\begin{notation}
    The partition $\lambda=a_1 + a_2 + \cdots +a_k$ is denoted by $a_1,a_2,\ldots,a_k$.
     In the case of a repeated term,
    we use frequency notation to indicate
    the number of terms.
    Thus, the $7$-part partition $\lambda=1+1+1+5+8+8+8$ of $n=32$ will be denoted by
     $1^3,5,8^3$ or just $1^35\,8^3$.
     Then, $\Omega(\lambda)=(0^34\,7^3)7 =4\,7^4.$
\end{notation}

For the triangular number
\begin{equation*}
    n=\binom{k+1}{2}=\binom{k}{2}+k=\big(1+2+\cdots + (k-1)\big)+k,
\end{equation*}
the partition $1\, 2\, 3\, \ldots\, k$ is a fixed point, and,  starting from
any initial partition of $n$, the trajectory will reach this fixed point.
For non-triangular numbers $n,$ there are no fixed points.
The complete characterization of the limit cycles is well known.
The recurrent states are parameterized by binary vectors  and the cycles are parameterized by binary necklaces.
This makes Theorem \ref{124cycles}, which is well known, immediate.

\begin{theorem} \label{CyclesTheorem} (\cite{AkinDavis1985, Brandt1982}). Suppose $n=\binom{k}{2}+r=1+2+\cdots +(k-1)+r$ where $1\le r \le k$.
     A partition $\lambda_1$ of $n$ is a recurrent state
    of Bulgarian solitaire if and only if it has the form
    \begin{equation*}
    \lambda_1=\delta_0, 1+\delta_1, 2+\delta_2,\ldots, k-1+\delta_{k-1},
    \end{equation*}
    where $\delta_i\in\{ 0,1\}$ and $\sum_{i=0}^{k-1} \delta_i =r$.
    Furthermore, the
    $\Omega$-move  $\lambda_1 \rightarrow \lambda_2=\Omega(\lambda_1)$
    has the effect of
    applying a cyclic shift to the corresponding binary vectors:
    \begin{equation*}
\vec{\Delta_1}=(\delta_0,\delta_1,\delta_2,\ldots,\delta_{k-1})
\rightarrow
\vec{\Delta_2}=(\delta_1,\delta_2,\ldots,\delta_{k-1},\delta_0).
    \end{equation*}
 Hence, the cycles are parameterized by binary necklaces,
 equivalence classes of binary vectors under cyclic shift.
\end{theorem}

We call the binary vector $\vec{\Delta}=\vec{\Delta}(\lambda)=(\delta_0,\delta_1,\delta_2,\ldots,\delta_{k-1})$ corresponding to a recurrent state an {\bf excess vector}. So, an excess vector is any binary vector with $r$ $1$'s for some $r>1.$  If $\vec{\Delta}(\lambda)$ has $\delta_0=0$, then $\lambda$ has only $k-1$ parts. For $n=\binom{k}{2}+r$ where $1\le r < k$, in every cycle there  are partitions with $k-1$ parts and the rest have $k$ parts.
Then, each cycle has (minimum) period a divisor $d>1$ of $k$, where $d$ is the minimum cyclic period of $\vec{\Delta}(\lambda)$. The first $d$ bits repeat $\frac{k}d$ times and $r=s\cdot\frac{k}{d}$ where $0<s<d$ is the number of $1$'s in those first $d$ bits.  In the triangular case, when $r=k$,
 the only possible excess vector is $(1,1,\ldots,1)$.

\begin{remark} \label{re:manycomponents}
    Note that for $n=\binom{k}2+r$, there are $\binom{k}{r}$ recurrent states in Bulgarian solitaire and each belongs to a cycle of length $d$ for some divisor $d$ of $k$. Hence, the number of connected components is bounded below by$\left \lceil \frac1k  {\binom{k}{r}} \right \rceil.$
\end{remark}

Halfway between the triangular numbers $\binom{k}{2}$ and $\binom{k+1}{2}$ is the {\bf half-square}

\begin{equation*}
    \left\lfloor \frac{k^2}{2} \right\rfloor=\left\{
  \begin{array}{cc}
      2t^2 &\text{ if $k=2t$}\phantom{+1..}\quad \\
      2t^2+2t &\text{ if $k=2t-1$}.\quad
      \end{array}
\right.
\end{equation*}

The triangular numbers start $0,1,3,6,10,15,21\ldots$ and halfway between them are the half-squares $0,2,4,8,12,18\ldots$\ .  Note that for $k=3$, the half-square \newline  $\left\lfloor \frac{k^2}{2} \right\rfloor=4$ is also a square. This happens for the solutions $k=1,3,17,\ldots$ of the Diophantine equation $k^2=2j^2+1.$ The half-squares and adjacent values appear in
Theorem \ref{separated}. The half-squares play an essential role in Section \ref{Sec_Par4}.

 For future use, we note three well-known applications of this correspondence between cycles in Bulgarian solitaire with minimum period $d$ and excess vectors with minimum cyclic period $d.$

 \begin{theorem} \label{124cycles}
 In Bulgarian solitaire,
 \begin{enumerate}
     \item The only fixed points  are for a triangular number $n=\binom{t+1}2=\binom{t}2+t$ for $t\ge 1$, and the only fixed point for that $n$ is the partition $1\, 2\,3\, \ldots\, t$ corresponding to the length $t$ excess vector $\vec{\Delta_0}=(1,1,\ldots)$. For any partition $\lambda$ of $n$, the trajectory starting at $\lambda$ ends at this fixed point.

     \item The only $2$-cycles  occur for a half-square $n=2t^2=\binom{2t}2+t$ and the only such  cycle is

    \begin{equation*}
        1^2 3^2 \ldots (2t-1)^2 \leftrightarrow
        2^2 4^2 \ldots (2t-2)^2 (2t),
    \end{equation*} corresponding to the length $2t$ excess vectors

 \begin{equation*} \vec{\Delta_1}=(1,0,1,0,\ldots) \leftrightarrow \vec{\Delta_2}=(0,1,0,1,\ldots).\end{equation*}

\item The only $4$-cycles are for \begin{itemize}
    \item $n=8t^2-t=\binom{4t}2+t,$
    \item $n=8t^2 = \binom{4t}2+2t,$ and
    \item $n=8t^2+t=\binom{4t}2+3t$,
\end{itemize}

with one $4$-cycle for each such $n$ corresponding,
respectively,  to the following cycles of length $4t$ vectors, where the first $4$ bits of each vector repeat $t$ times:

 \begin{itemize}
     \item $ (1,0,0,0,\ldots) \rightarrow  (0,0,0,1,\ldots) \rightarrow   (0,0,1,0, \ldots) \rightarrow (0,1,0,0, \ldots)  \rightarrow$\ ,
     \item $ (1,1,0,0,\ldots) \rightarrow  (1,0,0,1,\ldots) \rightarrow   (0,0,1,1, \ldots) \rightarrow (0,1,1,0, \ldots)  \rightarrow$\ , and
      \item $ (1,1,1,0,\ldots) \rightarrow (1,1,0,1, \ldots)  \rightarrow (1,0,1,1, \ldots)\rightarrow (0,1,1,1, \ldots) \rightarrow \ .$
 \end{itemize}

    \end{enumerate}
\end{theorem}

\begin{proof} Among nonzero binary vectors,
    \begin{enumerate}
        \item The only one with period $1$  of length $t$ is $\vec{\Delta_0}$,  and it is the only possible excess vector for $n=\binom{t}2+t$.
        \item The only ones with minimum period $2$ and length $2t$ are $\vec{\Delta_1}$ and $\vec{\Delta_2}$, which are cyclic shifts of each other.
        \item The only ones of period $4$ and length $4t$ are those given.
    \end{enumerate}
\end{proof}

 In contrast to the characterization of the recurrent states, much less is known about the trajectories from the transient states to the cycles. In the triangular case, the fixed point $1\,2\,3\ldots t$ is known to be reached in at most $t(t-1)$ moves \cite{Igusa1985}, a result from 1985. The general case remains a topic of active research \cite{Pham2023}.

 We call a partition $\lambda$ of $n$ {\bf separated} if there are no consecutive parts in it and define a {\bf defect} of $\lambda$ to be a $j>0$ so that
  $\lambda$ has parts of sizes $j$ and $j+1.$ So, a partition is separated exactly if it has no defects. The sequence giving the number of separated partitions of $n$,   i.e. ``partitions of $n$ in which any two distinct parts differ by at least $2$'', is sequence \href{https://oeis.org/A116931}{A116931}  in the On-Line Encyclopedia of Integer Sequences (OEIS) \cite{Sloane99}.
  The partitions just seen in the cycles of period $2$ are separated.  It turns out that for most $n$, there are no separated partitions in cycles.

\section{ Overview of the paper}

In Bulgarian solitaire, for most $n$ there are no separated partitions in cycles, each one has defects. As shown in Theorem \ref{separated}, recurrent partitions without defects occur only at the half-squares  $n=2t^2=\binom{k}{2}+r=\binom{2t}{2}+t$ and $2t^2+2t=\binom{2t+1}{2}+t$, as well as at the adjacent values $2t^2-1$ and $2t^2+2t+1.$ Then, of the $\binom{k}{r}$ recurrent partitions of $n$, one or two are defect-free. Those that are separated are variations of the partitions in $2$-cycles. Recall that each half-square is midway between two triangular numbers $\binom{k}2$ and $\binom{k+1}2$.  As $n$ moves away from a half-square and closer to the bounding triangular numbers, the least number of defects of any recurrent partition increases. In Section \ref{sec:defects}, we justify these claims about separated partitions and find all that do occur. However, these results are only for motivation and are not used elsewhere.

With this motivation, in Section \ref{SecFS}, we modify the rules of Bulgarian solitaire
to arrive at a variant, Floridian solitaire, in which the player has choices, but only separated partitions are allowed.
So, the player may end up with no moves and lose.
On each turn, the player makes two moves.
Starting from $\lambda_1,$ the first move, called an $\alpha$-move,
is to choose $s \ge 1$ parts, reduce them by $1$, and add a new $s$-part,
resulting in $\lambda_2$.
The second move is to $\lambda_3=\Omega(\lambda_2).$

In Definition \ref{Note_G_n}, we give a directed bipartite graph $G_n$ that encodes the possible turns in Floridian solitaire.  The two moves of a turn correspond to a pair of arcs $\lambda_1 \rightarrow \lambda_2$ and $\lambda_2\rightarrow \lambda_3$ forming a path. We also define a graph $H_n$ with an arc $\lambda_1 \rightarrow \lambda_3$ for each such path.

In  Section \ref{Sec_Computations}, we describe computational results for the graphs $G_n$ and $H_n$ for $n<100.$ We do not rely on these claims in the rest of the paper, but they do explain some of the motivation for the rest of the paper.  There are cycles in $G_n$ for all even $n \le 100$ except $n=4,10,12,20,22$,  and for $n=11$ and all odd $21 \le n \le 99$. The graph induced by the cycles seems to be large and highly connected for  $n>40$. This very strongly suggests that, with the few exceptions noted, every graph $G_n$ has cycles.  Our main result is that $G_n$ has at least one cycle for every $n>73$.
Every cycle in $G_n$ has even length.  Theorem \ref{FS2cycles} shows that $2$-cycles in Floridian solitaire are rare. However, unlike Bulgarian solitaire, there are many $4$-cycles for all large enough $n$.

Our main result, Theorem \ref{thm4cyclesInGn},
asserts that there is always a $4$-cycle for $n>73.$
We prove this using several parametric families. In Section  \ref{Sec_cyc4}, we provide some notation and examples, including a family of $4$-cycles $Q(x,y,z)$.

 In Section \ref{Sec_Par4}, we show in Proposition \ref{prop4cyclesINnMostGn} that $G_n$ has a $4$-cycle $Q(x,y,z)$ for most $n$. The exceptions are all close to a half-square. To prove Theorem \ref{thm4cyclesInGn}, we then provide $16$ one-parameter families that  cover the values $n>73$ missed by the family $Q(x,y,z)$, as well as some smaller $n$.  The methods used to provide $4$-cycles for these exceptional cases are applicable to cycles of arbitrary even length. We conclude with some open questions.

 \section{ Defects in Recurrent Partitions of Bulgarian Solitaire} \label{sec:defects}

We now justify the claims made about separated recurrent partitions of Bulgarian solitaire. These results are for motivation and are not used elsewhere. Theorem \ref{CyclesTheorem} makes it routine to find the separated partitions that occur in cycles and to find the least defects possible for a recurrent partition of a given $n.$ This is done in Theorems \ref{separated} and \ref{th:mindefects}.

We first consider the fixed points, $2$-cycles, and $4$-cycles from Theorem \ref{124cycles}.
\begin{itemize}
 \item  The fixed points are not separated; they have $t$ parts, $t-1$ defects, and length $t$ excess vector $(1,1,\ldots)$.
  \item The 2-cycles
\begin{equation*}
        1^2 3^2 \ldots (2t-1)^2 \leftrightarrow
        2^2 4^2 \ldots (2t-2)^2 (2t)
    \end{equation*}
    are the only cycles with all members separated.
    This is related to the fact that the  excess vectors with cyclic period $2$,
    \begin{equation*} (1,0,1,0,\ldots) \leftrightarrow (0,1,0,1,\ldots) \end{equation*}
    are exactly the non-constant binary vectors all of whose shifts have alternating entries.
\item The three 4-cycles with excess vector of length $8,$ for $n=30,32,$ and $34,$ are
 \begin{itemize}
    \item $1^2 2\,3\,5^2 6\,7 \rightarrow 1\,2\,4^2 5\,6\,8 \rightarrow
 1\,3^2 4\,5\,7^2 \rightarrow 2^2 3\,4\,6^2 7 \rightarrow,$
    \item $1\, 2^2 3\,5\, 6^2 7 \rightarrow 1^2 2\, 4\, 5^2 6\,8 \rightarrow
 1\,3\, 4^2 5\, 7\,8 \rightarrow 2\, 3^2 4\,6\,7^2 \rightarrow,
 \text{ and}$
    \item $1\,2\,3^2 5\,6\,7^2 \rightarrow 1\,2^2 4\,5\,6^2 8 \rightarrow
 1^2 3\,4\,5^2 7\,8 \rightarrow 2\,3\,4^2 6\,7\,8 \rightarrow$.
 \end{itemize}
 The partition $1\, 2^2 3\,5\, 6^2 7$ with excess vector $(1,1,0,0,1,1,0,0)$ has $4$ defects while the partition $1\,3\, 4^2 5\, 7\,8$  with excess vector $(0,0,1,1,0,0,1,1)$ has $3$ defects.
\end{itemize}

Looking at these partitions and their excess vectors illustrates that defects of a recurrent partition $\lambda$ correspond to the cases of $\delta_i=\delta_{i+1}$  in $\vec{\Delta}(\lambda)$. This includes $\delta_0=\delta_1=1$, but $\delta_0=\delta_1=0$ is excluded from the count. For $1 \le r <k$, a recurrent partition of $n=\binom{k}{2}+r$ has at most  $k-2$ defects. The unique partition with $k-2$ defects is the partition  $1\,2\,3\,\ldots\,k-1$ of $\binom{k}{2}$ with an additional part of size $r$. The corresponding excess vector has all $r$ $1$'s at the start.

Theorem \ref{CyclesTheorem} gives a straightforward answer to the question of which separated partitions occur in a cycle of Bulgarian solitaire. Note that the separated partitions not in $2$-cycles are variations of the ones that are

\begin{theorem} \label{separated}
     The only separated partitions belonging to a cycle in Bulgarian solitaire are:
     \begin{itemize}
     \item The two partitions $1^2 3^2 \ldots (2t-1)^2$ and  $2^2 4^2 \ldots (2t-2)^2 (2t) $ of  $n=2t^2$. \newline They form a $2$-cycle.
     \item The partition $1\,3^2 5^2 \ldots (2t-1)^2 $ of $n=2t^2-1.$ \newline Its cycle has length $2t.$
      \item The two partitions $1\,3^2 5^2 \ldots (2t-1)^2 (2t+1)$ and $2^24^2\ldots(2t)^2$ of $n=2t^2+2t$.\newline They belong  to a cycle of length $2t+1$ and are consecutive.
     \item The partition $1^2 3^2 \ldots (2t-1)^2 (2t+1)$ of $n=2t^2+2t+1.$ \newline Its cycle has length $2t+1$.
      \end{itemize}
\end{theorem}

\begin{proof}

  Two positive parts, $i+\delta_i$ and $i+1+\delta_{i+1}$ differ by $1$ exactly if  $\delta_{i}=\delta_{i+1}.$
  So, a recurrent state is separated exactly if the entries of the excess vector \newline   $ \vec{\Delta}=(\delta_0,\delta_1,\delta_2,\ldots,\delta_{k-1})$, either alternate or alternate except that $\delta_0=\delta_1=0$. \\

  For $k=2t$ even, the possible excess vectors are: \begin{itemize}
   \item   $(0,1,0,1,\ldots,0,1)$  and its shift
$(1,0,1,0,\ldots ,1,0)$ with  $r=k-r=t$ as in the middle case of Theorem  \ref{124cycles}, and
   \item $(0,0,1,0,1,\ldots,0,1,0)$ with $r=t-1$ and $k-r=t.$

  \end{itemize}
For $k=2t+1$ odd, the possible excess vectors are: \begin{itemize}
   \item $(0,0,1,0,1\ldots,0,1)$ and its shift $(0,1,0,1,\ldots,0,1,0)$ with $r=t$ and $k-r=t+1$  as well as
    \item   $(1,0,1,0,\ldots,1)$ with $r=t+1$ and $k-r=t.$

  \end{itemize}
 In the cycles containing these separated partitions, the other partitions, if any, each have one defect.
\end{proof}

Suppose that the most defects among the partitions in a length $k$ cycle of Bulgarian solitaire is $d$. Then, the other partitions $\lambda$ of the cycle have either $d$ defects or $d-1$ defects in the case that $\vec{\Delta}(\lambda)$ has $\delta_0=\delta_1=0$ or $\delta_0=\delta_{k-1}=1$. This proof also makes it clear that if $r$ and $k-r$, the numbers of $1$'s and $0$'s in the excess vector, are not close, then every recurrent partition of $n$ has several defects.

\begin{theorem} \label{th:mindefects} Suppose that $n=\binom{k}{2}+r$ with $k,r\ge 1$ and let $u=|r-(k-r)|.$ Then, every recurrent partition of $n$ has at least $u-2$ defects. \end{theorem}
\begin{proof}

To minimize defects, we must maximize the number of instances of $\delta_i \ne \delta_{i+1}$ other than, perhaps,  $\delta_0=\delta_1=0.$
If $1 \le r < k-r-1$,  then the partitions with the least defects
are those whose excess vector begins $0,0$ and never has $1,1$.
These have $k-2r-2=u-2$ defects.
If $r> k-r-1$,
then the partitions with the least defects are those
whose excess vector begins with $1$ and never has $0,0.$
These have $2r-k-1=u-1$ defects.\end{proof}

\section{Floridian solitaire} \label{SecFS}

 We now modify the rules of Bulgarian solitaire to arrive at a variant in which all partitions are required to be separated. We will investigate the cycles in this game.  It is again a one-player game, but, unlike Bulgarian solitaire, the player has some choices but may be left with no moves and hence lose. Again, the positions are partitions of a fixed integer $n$. Each turn consists of two moves: an $\alpha$-move, then an $\Omega$-move, each of which must finish at a separated partition. The $\alpha$-move could be described as the $\Omega$-move applied to a subset of the parts. Given a partition $\lambda_1,$  the player will:

\begin{enumerate}
    \item Select $s \ge 1$ parts; then
    reduce the selected parts by 1 and add a new $s$-part,
    creating a partition $\lambda_2$ such that
    \begin{itemize}
        \item $\lambda_2\ne \lambda_1$,
        \item $\lambda_2$ is separated, and
        \item $\Omega(\lambda_2)$ is separated.
    \end{itemize} The requirement $\lambda_2 \ne \lambda_1$ simply means that when the player selects only a single part, that part is not a $1$-part. \label{item:alphamove}
    \item Apply the $\Omega$-rule, resulting  in $\lambda_3=\Omega(\lambda_2)$, which must be separated. \label{item:omegamove}
\end{enumerate}

 A particular partition $\lambda$ is called an $\alpha$-\textbf{position} when it occurs at the start of a turn and an $\Omega$-\textbf{position} when it occurs in the middle of a turn. In a few places, we write $[\lambda,\alpha]$  in the first case and $[\lambda,\Omega]$  in the second case.

In the spirit of Austrian solitaire, introduced in \cite{AkinDavis1985}, Carolina solitaire, introduced in \cite{GriggsHo1998}, and Montreal solitaire, introduced in \cite{Cannings1992}, we refer to this new game as \textit {Floridian solitaire}
because the first and third authors
reside in Florida.

 \begin{lemma} \label{PreSeparatedlemma} Suppose that $\lambda_2$ is a separated partition with $p$ parts and that $\Omega(\lambda_2)=\lambda_3.$ Then, $\lambda_3$ is separated exactly if  $\lambda_2$ has no parts of size $p$ or $p+2$. \end{lemma}

\begin{proof} When each part of $\lambda_2$ is reduced by $1$, the result is a partition of $n-p$ that is separated. Then, the new part of size $p$ creates $\lambda_3.$ So, $\lambda_3$ fails to be separated exactly if it has a $(p-1)$-part or $(p+1)$-part, which happens exactly if $\lambda_2$ has a $p$-part or $(p+2)$-part.\end{proof}

\begin{definition} \label{Note_G_n}
 Let  $S(n)$ denote the set of all separated partitions of $n$. Also, let $T(n)$ be the set of separated partitions of $n$ with the additional property of Lemma \ref{PreSeparatedlemma}:  no part has size $p$ or $p+2$ where $p$ is the number of parts.

 Let $G_n$ be the bipartite  digraph with  node set

 \begin{equation*}
     V(G_n) =V_{\alpha} \cup V_{\Omega},
 \end{equation*}
 where $V_{\alpha} = \{[\lambda,\alpha]: \lambda \in S(n)\}$
 and   $V_{\Omega} = \{[\lambda,\Omega]: \lambda \in T(n)\}$,
 and with arcs
 \begin{itemize}
    \item  $[\lambda_1, \alpha] \rightarrow[\lambda_2,\Omega]$
    provided that there is an $\alpha$-move from $\lambda_1$
    to $\lambda_2 \in T(n)$ and
    \item  $[\lambda_2, \Omega] \rightarrow [\Omega(\lambda_2),\alpha].$
 \end{itemize}

 Each node $[\lambda_1,\alpha]$ may belong to zero, one, or several arcs $[\lambda_1,\alpha] \rightarrow [\lambda_2,\Omega]$. However, each node $[\lambda_2,\Omega]$ belongs to  one arc $[\lambda_2,\Omega] \rightarrow [\Omega(\lambda_2),\alpha].$ Note that a successful turn $[\lambda_1, \alpha] \rightarrow[\lambda_2,\Omega] \rightarrow [\lambda_3,\alpha]$ requires both $\lambda_2$ and $\lambda_3$ to be separated. From  Lemma \ref{PreSeparatedlemma}, this happens exactly when $\lambda_2 \in T(n).$

As before, $\lambda \rightarrow \Omega(\lambda)$ is the $\Omega$-move from $\lambda.$ This is the second move of a turn in Floridian solitaire. The first move of a turn is an $\alpha$-move. An equivalent definition of an $\alpha$-move from $\lambda_1$ is that some parts are chosen, and the rest are removed, creating a partition $\lambda^{*}_1$ of some $2 \le n'\le n.$ The $\Omega$-move is applied, producing $\lambda^{*}_2=\Omega(\lambda^{*}_1)$, and then the removed parts are restored yielding $\lambda_2.$ There is the additional requirement that $\lambda_2 \in T(n).$  Set $\alpha(\lambda_1)=\{\lambda_2 : \text{ there is an }\alpha\text{-move from }  \lambda_1 \text{ to  }\lambda_2 \in T(n)\}$ and  $\alpha(S)=\bigcup_{\lambda \in S}\alpha(\lambda).$

If the player cannot complete a turn, they lose. A {\bf winning position} is a $\lambda$ such that, from $[\lambda,\alpha],$  the player can avoid losing with proper play. A node $[\lambda,\alpha]$ is a {\bf recurrent state} if it belongs to a cycle in $G_n.$ Then, we also call $\lambda$ a recurrent state.
 So, $[\lambda,\alpha]$ is a winning position  if there is a path from it to a recurrent state. An {\bf immediate loss} is an $\alpha$-state from which there are no legal moves. So, $\lambda$ fails to be a winning position exactly if every path from $[\lambda,\alpha]$ in  $G_n$ ends at an immediate loss.

 We also define a related directed graph $H_n$. The nodes are the partitions in $S(n)$  and there is an arc $\lambda_1 \rightarrow \lambda_3$ exactly if there is a $2$-arc path $[\lambda_1,\alpha] \rightarrow [\lambda_2,\Omega] \rightarrow [\lambda_3,\alpha]$  in $G_n.$ Each cycle in $G_n$ has some even length $2k$ and  corresponds to a cycle of length $k$ in $H_n$. The recurrent states are the same in both graphs. We call the subgraph of $H_n$ induced by the recurrent states, the  {\bf core}. To find the core, we can restrict to the subgraph induced by $\Omega(T(n))=\{\Omega(\lambda),\lambda \in T(n)\}$.
\end{definition}

\begin{remark}
    It turns out that for $n=2$ and $n=6$, every state is a winning state, so there are no immediate losses. It seems certain that these are the only cases. It is easy to show that there are immediate losses for $n=2m+1$ and for $n=3m+1$. For $n=3$ and $n=4$, $3 $ and $2^2$ are immediate losses. The partition $1\,m^2$ of $n=2m+1>3$ is an immediate loss because the only $\alpha$-moves are to $1\,2\,(m-1)^2$, which is not separated, and to $3\,(m-1)^2$, which is separated
    (for $n\ne 7,11$),
    but not in $T(n)$ because it has $p=3$ parts and a $p$-part. More directly, $$\Omega(3\,(m-1)^2)=2\,3\,(m-2)^2.$$ Similarly, the partition $ 1\,3^m$ of $n=3m+1>4$ is an immediate loss because the only $\alpha$-moves are to $ 1\,2^m\,m,$ which is not separated, and to $2^m(m+1)$, which is not in $T(n).$
\end{remark}

A particular $\alpha$-move from $\lambda_1$ may be specified by indicating the parts that change. We next give a compact notation for doing this.
\begin{notation} \label{note_alphamove}
During an $\alpha$-move, when an $m$-part for $m>1$ is reduced, every \mbox{$m$-part} must be reduced. Otherwise, the resulting position will have two parts whose sizes differ by one. If there are several $1$-parts, it is possible to reduce (i.e. remove) only some of them and leave the rest.
   To specify an $\alpha$-move, we will underline the part sizes that are reduced
    by the specified $\alpha$-move.
    In the case of $1$-parts,
    we will  underline the $1$ and use a subscript  to
    indicate the number of $1$-parts removed.
\end{notation}

The result of an $\alpha$-move from $\lambda_1$ is $\lambda_2=\alpha(\lambda_1)$.  Given the corresponding $\alpha$-arc  $\lambda_1 \rightarrow \lambda_2$, it is not strictly necessary  to indicate the parts that  changed; they are apparent from carefully comparing $\lambda_1$ to $\lambda_2$. For clarity, we usually do indicate them with underlines and, in the following examples, also describe them.

\begin{figure}[h]
    \centering
    \includegraphics[width=1\linewidth]{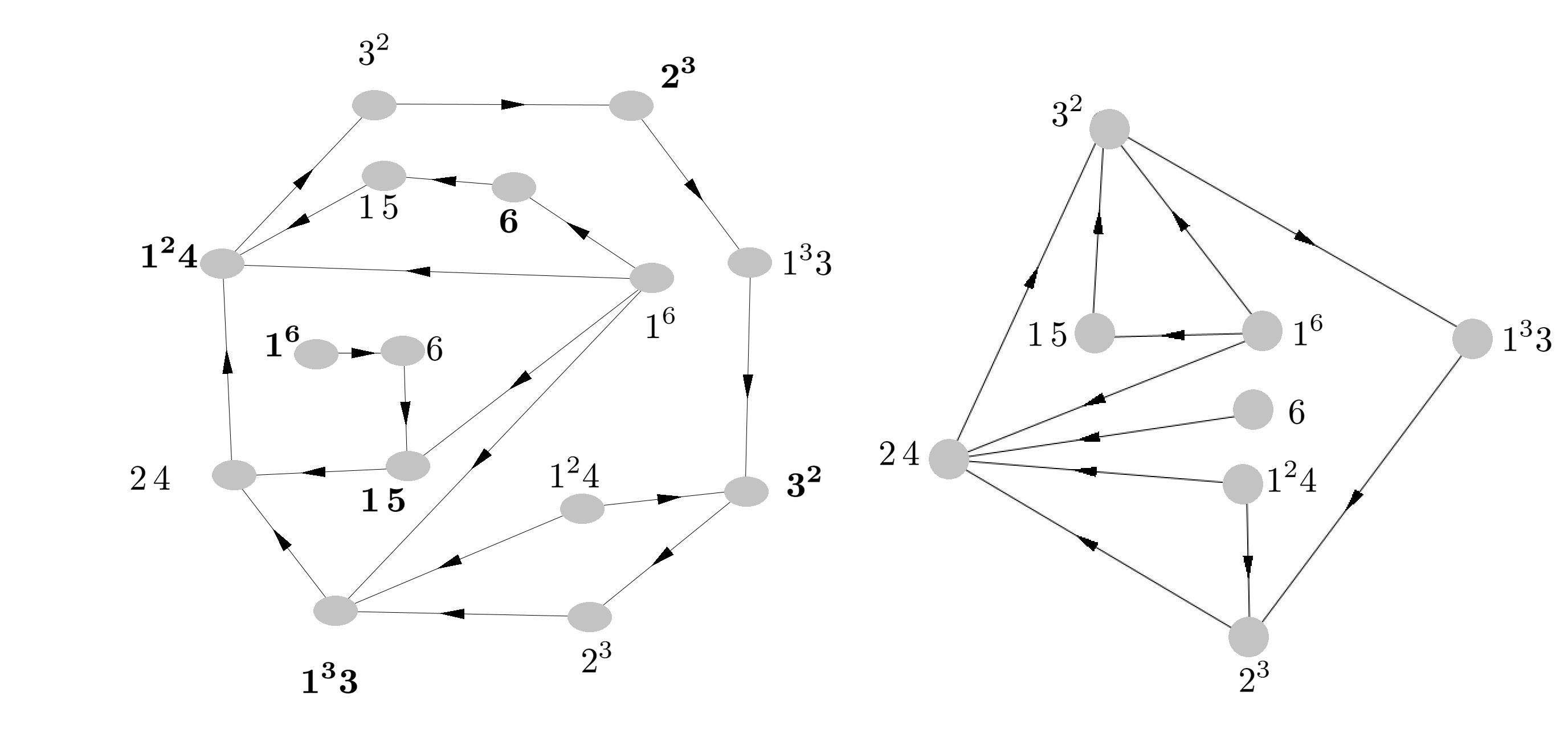}
    \caption{The digraphs $G_6$ and $H_6$.}
    \label{fig:G6andH6}
\end{figure}

\begin{example} \label{G6H6}  \normalfont $n=6$ has  eleven partitions, of which eight are in $S(6)$. Of these, $2\,4 \notin T(6),$  but the other seven are in $T(6)$.   The graphs $G_6$ and $H_6$ are shown in  Figure \ref{fig:G6andH6}. The seven $\Omega$-positions are shown in bold.

For four of the $\alpha$-positions, the only  $\alpha$-move is the $\Omega$-move. The exceptions are \begin{itemize}
    \item $2\,4$, with the move $\underline{2}\,4 \rightarrow \mathbf{1^24}$,
    \item $1^24$,  with the moves $1^2\underline{4} \rightarrow \mathbf{1^33}$ and the $\Omega$-move $\underline{1}^2\underline{4} \rightarrow {\mathbf 3^2}$,
    \item $1^33$, with the move $\underline{1}_3^33\rightarrow {\mathbf 3^2},$ and
    \item $1^6$, from which one can take $3,4,5,$ or all $6$ of the parts, resulting in\newline  $\mathbf{1^33,1^24,1\,5},$ and ${\mathbf 6}$ respectively.
\end{itemize}

The nodes of $H_6$ have the same out-degree as the corresponding nodes of $G_6$ except that $1^6$ has only three arcs leaving it in $H_6$, because the paths $1^6 \rightarrow \mathbf{1\,5} \rightarrow 2\,4$ and  $1^6 \rightarrow \mathbf{1^33} \rightarrow 2\,4$ have the same ends.

The only cycle of $G_6$ is the $8$-cycle

$$\underline{2}\,4\rightarrow \mathbf{1^24} \rightarrow \underline{3}^2 \rightarrow \mathbf{2^3} \rightarrow \underline{1}^3_33 \rightarrow \mathbf{3^2} \rightarrow \underline{2}^3 \rightarrow \mathbf{1^33} \rightarrow\ .$$
So the core of $H_6$ is the $4$-cycle $$2\,4 \rightarrow 3^2 \rightarrow 1^33 \rightarrow 2^3 \rightarrow\ .$$

 \end{example}

We could  discard $\Omega$-positions in $T(n)$ that cannot be reached by an $\alpha$-move, for example $1^6$ for $n=6$. That is, we could use $\alpha(S(n))$ instead of $T(n)$.  Any member of $S(n)$ could be an initial position. In later turns, every $\alpha$-position is in $\Omega(T(n)).$ Our main interest is the recurrent states and cycles. For this purpose, we could restrict to the subgraph of $G_n$ induced by $V_{\Omega} \cup \{[\Omega(\lambda),\alpha] :\lambda \in T(n)\}.$   We illustrate this for $n=24$ and $n=25$.

\begin{example} \label{cts25}  \normalfont Of the  $1575$ partitions of $n=24$,  $|S(24)|=424$ are separated and  $|T(24)|=326$ of these are such that the $\Omega$-move takes them to $S(24)$.  The set  $\Omega(T(24))$ has $189$ members. Any of the $424$ separated positions could be an initial position for the game. If play continues, the player will only face one of these $189$ positions.  The core of $H_{24}$ is strongly connected with $35$ nodes and $60$ arcs.

Of the  $1958$ partitions of $n=25$, $|S(25)|=472$ are separated and $|T(25)|=315$ of these are such that the $\Omega$-move takes them to $S(25)$. The set  $\Omega(T(25))$ has $172$ members. Any of the $472$ separated partitions could be an initial position for the game. If play continues, the player will only face one of those $172$ positions. The core of $H_{25}$ has $16$ nodes and $19$ arcs. This core is shown in  Figure \ref{fig:11212325}.
 \end{example}

 Each $2$-cycle of Bulgarian solitaire appears as a $2$-cycle of Floridian solitaire in two ways, for the same $n.$
 We now show that these are the only $2$-cycles.

  \begin{theorem} \label{FS2cycles} The graph $G_n$ has $2$-cycles exactly if $n=2t^2$. Then, there are just two $2$-cycles,  $[\lambda_1,\alpha] \leftrightarrow [\lambda_2,\Omega]$ and $[\lambda_2,\alpha] \leftrightarrow [\lambda_1,\Omega]$, where  $$
        \lambda_1=1^2 3^2 \ldots (2t-1)^2 \leftrightarrow
        \lambda_2 =2^2 4^2 \ldots (2t-2)^2 (2t)
  $$ is the $2$-cycle of Bulgarian solitaire for that $n$.  The graph $H_n$ has loops exactly if $n=2t^2$. Then, there are just the loops on $\lambda_1$ and $\lambda_2.$ \end{theorem}

\begin{proof}
 The result for $H_n$ follows from that for $G_n.$ Let $[\lambda_1,\alpha] \leftrightarrow [\lambda_2,\Omega] $ be a $2$-cycle in $G_n$. Then, both partitions are separated and $\Omega(\lambda_2)=\lambda_1$. If we show that the $\alpha$-move from $\lambda_1$ to $\lambda_2$ reduces every part, then it is the $\Omega$-move and the result follows from Theorem \ref{124cycles}.

 The partitions of $n=2$, namely $2$ and $1^2,$ each allow only the $\Omega$-move. So, assume that $n>2.$  Then, $\lambda_2\ne n$, since that would force $\lambda_1=\Omega(n)= 1\,(n-1)$ from which there is no $\alpha$-move to $n$. We wish to show that the $\alpha$-move reduces every $m$-part of $\lambda_1$. Consider first the case  $m>1$. Since there is an $m$-part in $\lambda_1,$  there are no $(m-1)$-parts in $\lambda_1=\Omega(\lambda_2)$, which means that there are no $m$-parts in $\lambda_2.$ Hence, the $\alpha$-move reduced every $m$-part of $\lambda_1$. Now consider the case of $1$-parts in $\lambda_1$. As $\lambda_2 \ne n$, the new part is not a $1$-part. So, there must be a $2$-part in $\lambda_2$ and hence no $1$-parts. Thus, all the $1$-parts of $\lambda_1$ were also removed in the $\alpha$-move.
 \end{proof}

\section{Observations about Floridian solitaire for
$n \le 100$}\label{Sec_Computations}

We describe the results of computations of $G_n$ and $H_n$ where $n\le 100.$ For small $n$, we computed the entire graph $G_n$. For moderate $n$, we restricted the computation to the portion of $H_n$ that could include any cycles. For large $n,$  we considered only the portion of $H_n$ that could contain $4$-cycles.  We do not rely on these claims in the rest of the paper, but they do explain some of the motivation.

In cycles of Bulgarian solitaire,  the new part created by an $\Omega$-move is always the largest of the resulting partition. This is not always the case in Floridian solitaire, but it seems to happen more often than not.  For $n=2t^2$, each of the two loops is an isolated component of $H_n.$ It seems highly likely that the core of $H_n$ is not only strongly connected for all $n>33$ (with the exception of the loops) but that the number of nodes and the ratio of arcs to nodes increase rapidly (but not monotonically) with $n$. The cores of $H_{18}$ and $H_{33}$ each have a small component and a large component. For larger $n \le 60$, the core is strongly connected, aside from loops. The number of losing positions seems to grow rapidly, but not nearly as rapidly as the number of separated partitions. The chances of being able to reach the strongly connected components from a random separated initial position appear to be high and increasing after $n=30$, although more slowly for the odd case than the even case.

The case $n=1$ is ruled out. There are  cycles for $n=2,6,8,11,14,16,18,21$, but none for the other $n<23$. For $23 \le n \le 100$, there are cycles. Together with this claim, Theorem \ref{thm4cyclesInGn}
establishes that there are cycles in $G_n$ for all $n \ge 23.$ It seems that the core of $H_n$, the subgraph induced by the recurrent states, becomes large and highly connected as $n$ increases. However, this happens more slowly in the odd case than in the even case.

We first consider the case where $n$ is even. For $n=4,10$, and $12$, there are no recurrent states. From Theorem \ref{FS2cycles}, $G_2$  has two $2$-cycles and $H_2$ has loops on $1^2$ and $2$. Those are the entire graphs, so every position is a winning position.

As we saw in Example \ref{G6H6},  for $n=6$, there are again no losing positions, although this almost certainly never happens again. There are $4$ recurrent states and they belong to an $8$-cycle in $G_6$. From Theorem \ref{FS2cycles}, $G_8$  has two $2$-cycles. These are the only cycles, so $H_8$ is just the two loops on $1^23^2$ and $2^24$. There are no other recurrent states. In contrast, Bulgarian solitaire for $n=8$ also has a $4$-cycle that includes $ 1\,3\,4$ with excess vector $(0,0,1,1).$ Of the $22$ partitions of $n=8$, thirteen are in $S(8).$  Neither $4^2$ nor $2\,6$ are in $T(8),$ but the other eleven separated partitions are, and they result in eight $\alpha$-positions that can occur after the initial turn. The partition $4^2$ is an immediate loss, but the other twelve separated partitions are winning positions.

The cores of $H_6,H_{14}$, and $H_{16}$ are shown in Figure \ref{fig:61416}.

 \begin{figure}[h]
    \centering
    \includegraphics[width=1\linewidth]{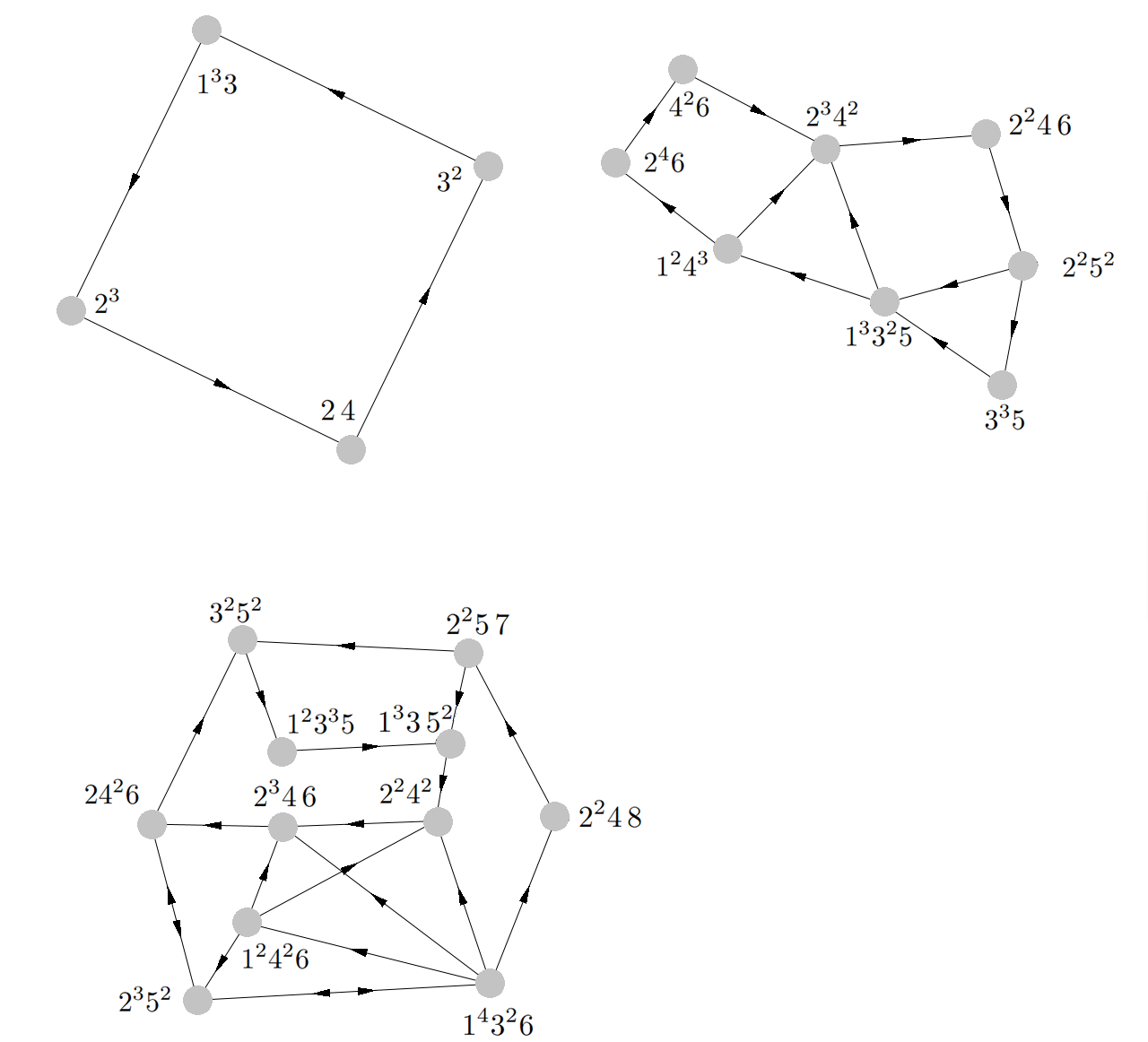}
    \caption{The cores of $H_6$, $H_{14},\text{ and }H_{16}$.}
    \label{fig:61416}
\end{figure}

The shortest cycles in $H_6$ and $H_{14}$ are $4$-cycles corresponding to $8$-cycles in $G_6$ and $G_{14}$. In $H_{16},$ the arc pairs $2^35^2 \leftrightarrow 1^43^26$ and $2^35^2 \leftrightarrow 2\,4^26$  correspond to a pair of $4$-cycles in $G_{16}$, namely

$$2^3\underline{5}^2 \rightarrow \mathbf{2^44^2} \rightarrow \underline{1}^4_33^26 \rightarrow  \mathbf{1\,3^36} \rightarrow$$  and
$$\underline{2}^35^2 \rightarrow \mathbf{1^33\,5^2} \rightarrow \underline{2}\,\underline{4}^26 \rightarrow  \mathbf{1\,3^36} \rightarrow.$$

From Theorem \ref{FS2cycles}, $G_{18}$  has two $2$-cycles and $H_{18}$ has two loops. The core of $H_{18}$ has four connected components. These are shown in Figure \ref{fig:18}. There are no recurrent states for $n=20$ and $n=22.$ As mentioned in Example \ref{cts25},  the core of $H_{24}$ is strongly connected with $35$ nodes and $60$ arcs.

 \begin{figure}[h]
     \centering
     \includegraphics[width=1\linewidth]{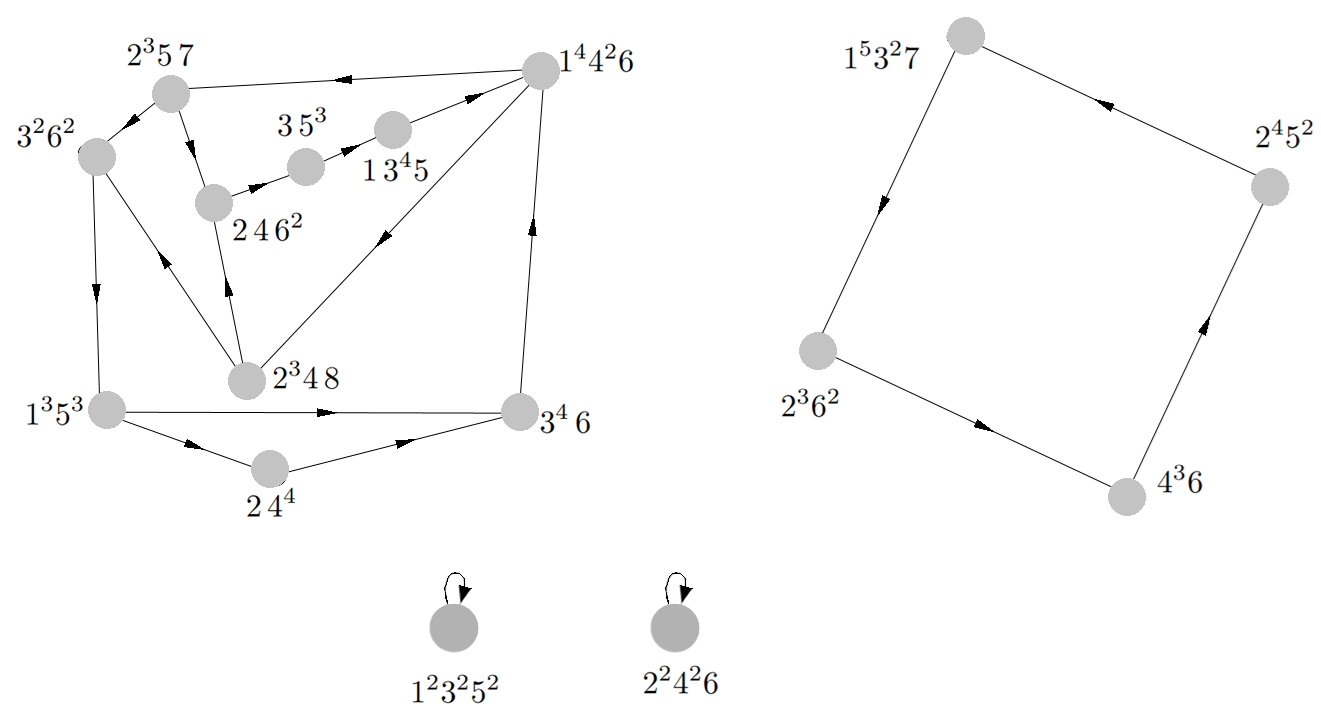}
     \caption{The core of $H_{18}$. }
     \label{fig:18}
 \end{figure}

 We now consider the case where $n$ is odd. $G_n$ and $H_n$ seem to be smaller for odd $n$ than for comparable even $n.$ Compare the cases of $n=24$ and $n=25$ described in Example \ref{cts25}. The cores of $H_n$ for $n=11,21,23,$ and $25$ are shown in Figure \ref{fig:11212325}. Those for $n=27,29,$ and $31$ are shown in Figure \ref{fig:27293133}, which also shows the smaller of the two components of the core of $H_{33}$ with four nodes and five arcs. The larger component has $53$  nodes and $91$ arcs. \newpage

   \begin{figure}[h]
     \centering
     \includegraphics[width=1\linewidth]{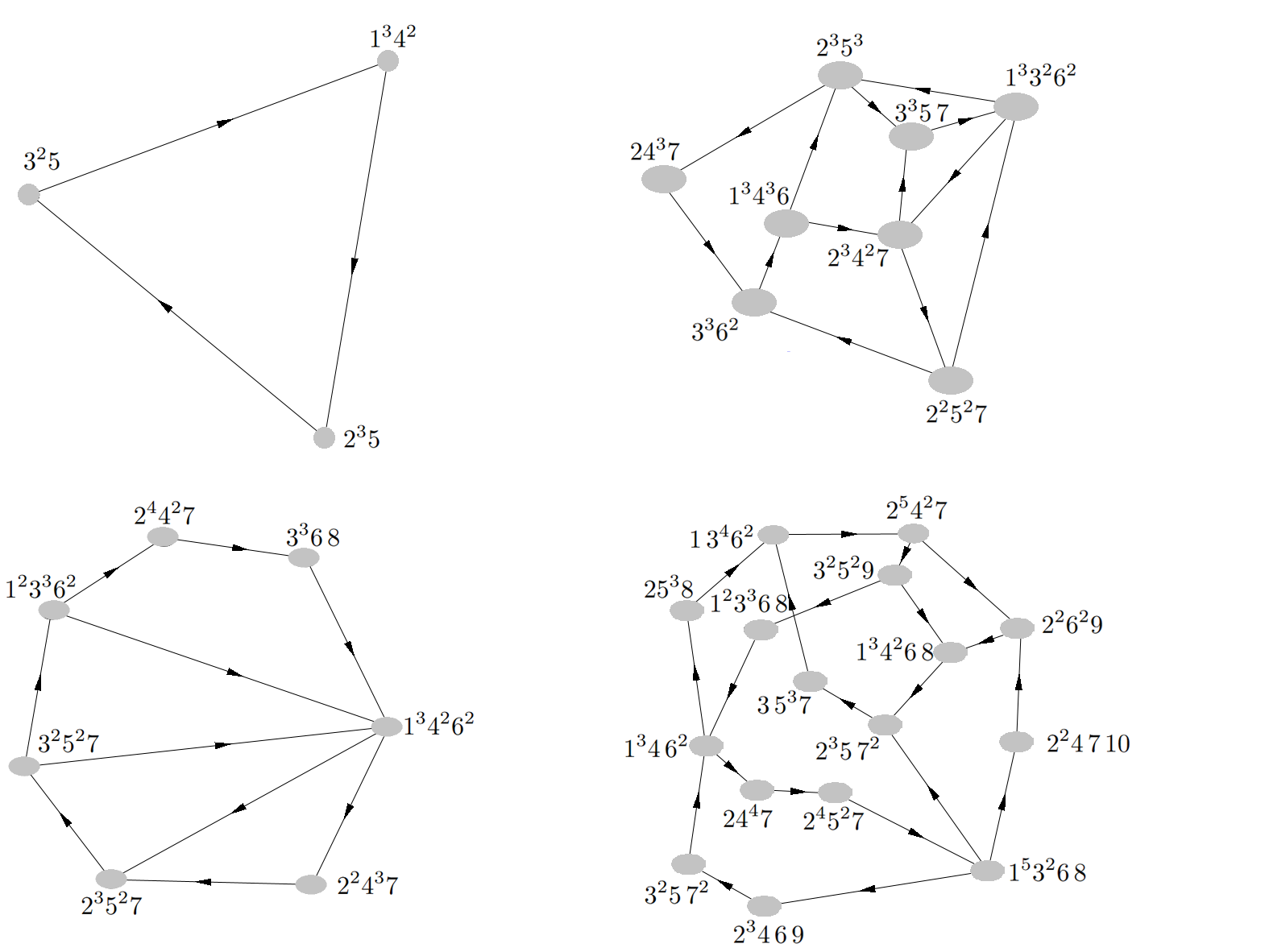}
     \caption{The cores of $H_n$ for $n=11,\, 21, \, 23$, and $25$.}
     \label{fig:11212325}
 \end{figure}

 \begin{figure}[h]
          \centering
          \includegraphics[width=1\linewidth]{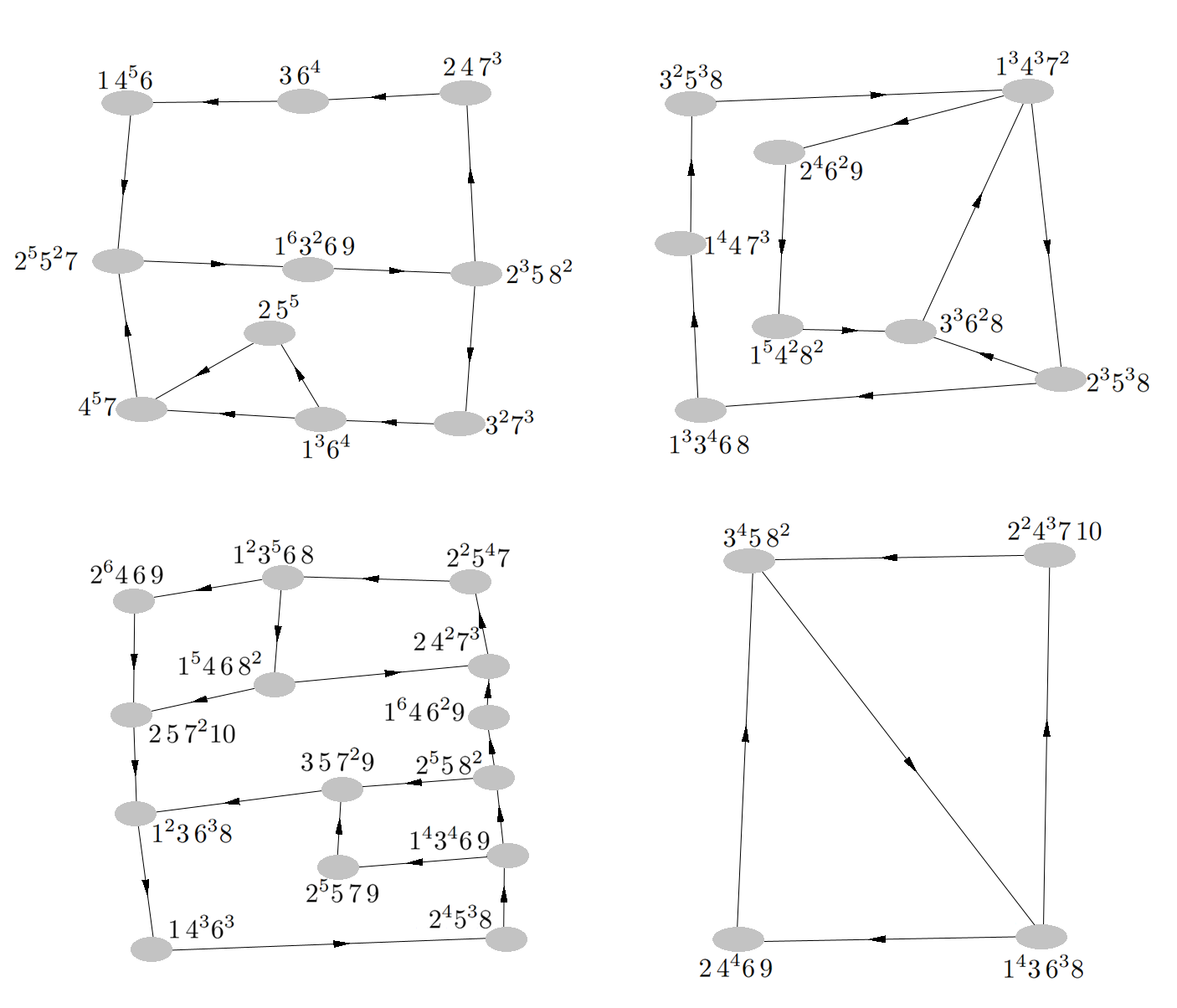}
         \caption{The cores of $H_n$ for
        $n=27,\,29,\,31$, and the smaller component of $H_{33}$.}
       \label{fig:27293133}
  \end{figure}

\section{ The 4-cycles in Floridian solitaire} \label{Sec_cyc4}

To prove, without appeal to computations, that in Floridian solitaire there are always cycles for large enough $n$, we turn to $4$-cycles.
In Bulgarian solitaire, there are few $4$-cycles.
From Theorem \ref{124cycles}, this happens only for $n=8t^2-t$,
$n=8t^2,$ and $n=8t^2+t$ with $t\ge 1.$
These $n$ have one $4$-cycle each.
So, there are three one-parameter families that together account for all the $4$-cycles of Bulgarian solitaire.

 \begin{remark} \label{re:no4cycs} Our computations show that, for $0\le  n \le 100$, $G_n$ has a $4$-cycle with the $35$ exceptions. These are the values $n\in \{0,1,\ldots,15\}$, along with the $19$ further exceptions of $n \in \{ 18, 20, 22, 36\}$ and $n \in \{ 17, 19, 21, 23, 25, 27, 29, 31, 33, 39, 41, 43,$ $49, 51, 73\}$.  We include the trivial cases of $n\in \{0,1\}$ and list these exceptions separately for the even and odd cases with $n>16$. This will be convenient in Section \ref{Sec_Par4} where we establish that there are $4$-cycles in  $G_n$ for all other $n.$
\end{remark} \newpage

In Floridian solitaire, $G_n$ has many $4$-cycles for large enough $n$. Our main result, Theorem \ref{thm4cyclesInGn}, is that there is a $4$-cycle in $G_n$ for all $n>73.$ With a little bookkeeping we strengthen this in Theorem \ref{thm4cyclesInGn_more} to show that there are $4$-cycles for $n\ge 16$ aside from $19$ values, the largest of which is $73.$    In this section, we provide some notation and examples in preparation for that result, including a three-parameter family  of $4$-cycles $Q(x,y,z)$. In Section \ref{Sec_Par4}, we show that for most $n$, $G_n$ has a cycle from this family. The exceptions are $8$  values of $n$ close to each half-square $2t^2$ or $2t^2+2t$, or fewer for small $t$. We provide $16$ one-parameter families that handle these exceptions.

Recall that a $4$-cycle in $G_n$ represents two turns for the player, each with an \mbox{$\alpha$-move} and an $\Omega$-move.
In the graph $H_n$, this corresponds to a $2$-cycle.

\begin{example} \label{simple} \normalfont

We provide $4$-cycles in $G_n$ for $n=16,30,$ and $48$ that are particularly simple and illustrate the relationship between them and
the corresponding $2$-cycles in $H_n$.
We indicate the $\Omega$-positions in bold and the $\alpha$-positions by underlines using  Notation \ref{note_alphamove}. These 4-cycles are

\begin{align*}
 &2^3\underline{5}^2\phantom{7^29^2}  \rightarrow&  &\mathbf{2^44^2}\phantom{7^29^2} \rightarrow&   &\underline{1}^4_33^26\phantom{8^210} \rightarrow&  & \mathbf{1\,3^36}\phantom{8^210} \rightarrow\ , \\
 &2^3\underline{5}^27^2\phantom{9^2}  \rightarrow&  &\mathbf{2^44^27^2}\phantom{9^2} \rightarrow&  &\underline{1}^4_33^26^28\phantom{10} \rightarrow&  & \mathbf{1\,3^36^28}\phantom{10} \rightarrow\ ,
 \text{ \ \ and } \\
  &2^3\underline{5}^27^29^2  \rightarrow&
  &\mathbf{2^44^27^29^2} \rightarrow&   &\underline{1}^4_33^26^28^210 \rightarrow&
  & \mathbf{1\,3^36^28^210} \rightarrow\ .
\end{align*}

We will call these $q_1,q_2$, and $q_3.$ The pattern turns out to continue and to give a $4$-cycle $q_z$  in $G_n$ for $n=2(z+2)^2-2.$
Here,  $\underline{5}^2$ indicates that both of the $5$-parts are reduced replacing $5^2$ with $4^2$. Since this is all that was done, this results in a fourth $2$-part. And $\underline{1}^4_3$ indicates that $3$ of the $1$-parts are removed, this results in a third $3$-part.

These 4-cycles in $G_n$ for $n=16,30,$ and $48$ become
2-cycles in $H_n$ for $n=16,30,$ and $48$.
These 2-cycles are

\begin{alignat*}{4}
 && 2^35^2\phantom{7^29^2} \quad && \leftrightarrow  \quad  && 1^43^26\ , \phantom{8^210,}
 \phantom{\text{ and }} \\
 && 2^35^27^2\phantom{9^2} \quad && \leftrightarrow \quad  && 1^43^26^28\ ,\phantom{10,} \text{ and }  \\
 && 2^35^27^29^2  \quad && \leftrightarrow \quad   && 1^43^26^28^210\ . \phantom{\text{ and }}
\end{alignat*}

\end{example}

\begin{example} \label{ex1_0_11} \normalfont

The cycle $q_{11}$  is given below. We no longer display the partitions for the $\Omega$-positions in bold since they are the partitions without underlines as well as the partitions where all parts change. This 4-cycle in $G_{336}$ is as follows, where the overbraces indicate the number of part sizes.

\begin{align*}
2^3, (\overbrace{\underline{5},7,9,\ldots,25}^{11})^2 \rightarrow\phantom{.} \\
2^4,4^2,(\overbrace{7,9,\ldots,25}^{10})^2 \rightarrow\phantom{.} \\
\underline{1}^4_3,3^2,(\overbrace{6,8, \ldots,24}^{10})^2(26) \rightarrow\phantom{.} \\
 1\,,3^3,(\overbrace{6,8,\ldots,24}^{10})^2(26) \rightarrow\phantom{.} \\
2^3, (\overbrace{\underline{5},7,9,\ldots,25}^{11})^2 \rightarrow.
\end{align*}

The fifth line is a repeat of the first to aid in following the cycle. In this case, the $\Omega$-positions and the $\alpha$-moves can be reconstructed from the $\alpha$-positions, but we will use this five-line format with the underline convention for ease of reading.

Here, $\underline{5}^2$ and $\underline{1}^4_3$ indicate the same things as in the previous example. This is the case for any $q_z.$
\end{example}

\begin{example} \label{4CycleExample}
 \normalfont

Here is a $4$-cycle in $G_{324}$.

\begin{align*}
    (\overbrace{2, 5, 8}^3)^{3},
    (\overbrace{10,12,14,16,18}^5)^{2},
     (\overbrace{\underline{21}, 23, 25}^3)^{2} \rightarrow \phantom{.} \\
     2^4, (\overbrace{ 5,8}^{2})^{3}, (\overbrace{10,12,14,16,18,20}^{6})^{2},
 (\overbrace{23,25}^{2})^{2} \rightarrow\phantom{.} \\
  \underline{1}^4_3,
 (\overbrace{ \underline{4,7}}^{2})^{3},
 (\overbrace{9,11,13,15,17,19}^{6})^{2},
 (\overbrace{22,24}^{2})^{2},
 \left(26\right) \rightarrow\phantom{.} \\
  1\,,
 (\overbrace{3,6, 9}^{3})^{3},
 (\overbrace{11,13,15,17,19}^{5})^{2},
 (\overbrace{22,24}^{2})^{2},\left(26\right)  \rightarrow\phantom{.} \\
 (\overbrace{2, 5, 8}^3)^{3},
    (\overbrace{10,12,14,16,18}^5)^{2},
     (\overbrace{\underline{21}, 23, 25}^3)^{2} \rightarrow.
\end{align*}

In the first partition, $\underline{21}$ means that the two $21$-parts are reduced.
This creates a fourth $2$-part.
In the third partition, $\underline{1}^4_3,  (\overbrace{ \underline{4,7}}^{2})^{3}$
means that a total of nine parts are reduced: three of the four $1$-parts
and the six parts with sizes $4$ and $7$. This results in a third $9$-part.
\end{example}

 \begin{example}  \label{4CycleExample_comp}

\normalfont Since the values in each pair of parentheses form an arithmetic progression, it is enough to indicate the first and last terms and the number of terms to indicate the progression, provided  that the parts affected by each $\alpha$-move are indicated. So,  the first  partition in the $4$-cycle above  will be written
 $$ (\overbrace{2\ldots8}^3)^{3},
    (\overbrace{10\ldots 18}^5)^{2},
     (\overbrace{\underline{21} \ldots 25}^3)^{2} $$

\end{example}

We now define a family $Q(x,y,z)$ of $4$-cycles. The cycles $q_z$ of  Examples \ref{simple} and \ref{ex1_0_11} are $Q(1,0,z)$ for $z=1,2,3,$ and $ z=11.$  The cycle of Example \ref{4CycleExample} is $Q(3,5,3)$. We will use  the compact form of Example \ref{4CycleExample_comp} in the definition.

\begin{definition}\label{notn:Q(x,y,z)}
    For $x\geq 1$, $y\geq 0$, and $z\geq 1,$ let $Q(x,y,z)$ be

    \begin{footnotesize}
    $ (\overbrace{2   \ldots 3x-1}^x)^{3},
    (\overbrace{3x+1  \ldots 3x+2y-1}^y)^{2},
     (\overbrace{\underline{3x+2y+2} \ldots 3x+2y+2z}^z)^{2}
     \rightarrow \\
      2^4, (\overbrace{ 5 \ldots 3x-1}^{x-1})^{3}, (\overbrace{3x+1  \ldots 3x+2y+1}^{y+1})^{2},
 (\overbrace{3x+2y+4  \ldots 3x+2y+2z}^{z-1})^{2} \rightarrow \\
   \underline{1}^4_3,
 (\overbrace{ \underline{4\ldots 3x-2}}^{x-1})^{3},
 (\overbrace{3x \ldots 3x+2y}^{y+1})^{2},
 (\overbrace{3x+2y+3 \ldots 3x+2y+2z-1}^{z-1})^{2},
 \left(3x+2y+2z+1\right) \rightarrow \\
   1\,,
 (\overbrace{3 \ldots 3x}^{x})^{3},
 (\overbrace{3x+2 \ldots 3x+2y}^{y})^{2},
 (\overbrace{3x+2y+3 \ldots 3x+2y+2z-1}^{z-1})^{2},\left(3x+2y+2z+1\right)  \rightarrow \\
 (\overbrace{2\ldots 3x-1}^x)^{3},
    (\overbrace{3x+1  \ldots 3x+2y-1}^y)^{2},
     (\overbrace{\underline{3x+2y+2} \ldots 3x+2y+2z}^z)^{2}.$
\end{footnotesize}

Also, let $n(x,y,z)=\tfrac{(3x+2y+2z)^2+3x+4z}{2}.$

\end{definition}

\begin{proposition} \label{prop4cycle3parameter1}
For  $x\geq 1$, $y\geq 0$, and $z\geq 1$,  $Q(x,y,z)$ is a $4$-cycle in $G_n$ for $n=n(x,y,z)=\tfrac{(3x+2y+2z)^2+3x+4z}{2}$.
\end{proposition}

\begin{proof}
    The value of $n$ is a simple computation. The two $\Omega$-moves reduce all the existing piles by $1$ and create a new pile that is the largest in the result. The first $\alpha$-move reduces $2$ parts and creates a fourth $2$-part. The second $\alpha$-move reduces $3x$ parts and creates a third $3x$-part.
\end{proof}

\section{Parametric families of 4-cycles in Floridian solitaire} \label{Sec_Par4}

In Theorems \ref{thm4cyclesInGn} and \ref{thm4cyclesInGn_more}, we show that there exists a 4-cycle in $G_n$ for all $n>73$ and for $39$ values in the interval $16 \le n \le 72.$ Recall that Theorem \ref{CyclesTheorem} used the form $n=\binom{k}{2}+r$ with $1\le r <k$, giving the positive distance up from the nearest triangular number. In discussing $4$-cycles in Floridian solitaire, we will use the forms $n=2t^2+2t+c$ for $0 \le c < 2t+1$ and $n=2t^2+c$ for $0 \le c \le 2t-1$ giving the non-negative distance up from the nearest half square.  Every nonnegative integer has a unique expression of this form. It turns out to be helpful to separate the even and odd subcases.

\begin{definition}\label{notn:subintervalAtandBt}
For $t\ge 0$, let $A_t=[2t^2+2t,2t^2+2t+(2t+1)]$ and for $t\ge 1$, let $B_t=[2t^2,2t^2+(2t-1)]$.
\end{definition}

So $A_0=[0,1],B_1=[2,3],A_1=[4,5,6,7],B_2=[8,9,10,11],\ldots .$ It is easy to see that these intervals partition the nonnegative integers.

Call $n$ {\bf exceptional} if $G_n$ does not have a $4$-cycle $Q(x,y,z).$
We show in Theorem \ref{thm4cyclesInGn_more} that, aside from the $35$ values listed in Remark \ref{re:no4cycs}, where the computations showed that there was not one, there is a $4$-cycle in $G_n$ for every exceptional $n.$

\begin{proposition} \label{prop4cyclesINnMostGn}
    There is a $4$-cycle $Q(x,y,z)$ in $G_n$ for exactly the following nonnegative values of $n$:
    \begin{enumerate}
        \item The values $n=2t^2+2t+c$ for even $4\le c \le 2t+1$. These are in $A_t.$
        \item The values $n=2t^2+c$ for odd $5\le c \le 2t-3$. These are in $B_t.$
         \item The values $n=2t^2+2t+c$ for odd $7 \le c \le 2t-3$. These are in $A_t.$
         \item The values $n=2t^2+c$ for even $8\le c \le 2t-6$. These are in $B_t.$
    \end{enumerate}

These include  $14$ values $n<73$
\begin{itemize}
    \item from  $A_2,A_3,A_4,$ and $A_5$ the $n\in \{16,\ 28,30,\ 44,46,48,\ 64,66,68,70\}$,
    \item from  $B_4$ and $B_5$ the  $n\in \{37,\ 55,57\}$, and
     \item from $A_5$ the value $n=67.$
\end{itemize}

    The probability that a randomly chosen $n \le N$ is exceptional is less than $\sqrt{\frac{128}{N}}$.
\end{proposition}

\begin{proof}

Recall that $Q(x,y,z)$ is in $G_n$ for $n=n(x,y,z)=\tfrac{(3x+2y+2z)^2+3x+4z}{2}.$
Observe that
$n(x+4,y,z) = n(x,y+3,z+3)$. So, we may restrict attention to $x \in \{1,2,3,4\}$. The nonnegative integers are partitioned into the intervals $A_t=[2t^2+2t,\ldots,2(t+1)^2-1]$ for $t\ge 0$ and $B_t=[2t^2,\ldots,2t^2+2t-1]$ for $t \ge 1$.    We now check which even and odd members of each interval have the values $n(x,y,z).$

\begin{enumerate}
    \item Let $x=1$ and $y+z=t-1$, with $y\ge 0$ and  $z \ge 1$.
Then $n(1,t-1-z,z)=2t^2+2t+2z+2$ for $z\in [1,t-1]$. These are the  values $2t^2+2t+c \in B_t$ for even $4 \le c \le 2t.$
  \item Let $x=2$ and $y+z=t-3$, with $y \ge 0$ and $z \ge 1.$
Then $n(2,t-3-z,z)=2t^2+3+2z$ for $z\in [1,t-3]$. These are the  values $2t^2+c \in A_t$ for odd $5 \le c \le 2t-3.$
 \item Let $x=3$ and $y+z=t-4$, with $y \ge 0$ and $t \ge 1.$
Then $n(3,t-4-z,z)=2t^2+2t+5+2z$ for $z\in [1,t-4]$. These are the values $2t^2+2t+c\in B_t$ for odd $7 \le c \le 2t-3.$
  \item Let $x=4$ and $y+z=t-6$, with $y \ge 0$ and $z \ge 1.$
Then $n(4,t-6-z,z)=2t^2+2z+6$ for $z\in [1,t-6]$. These are the values $2t^2+c \in A_t$ for even $8 \le c \le 2t-6.$
\end{enumerate}

The values $n(x,y,z)<73$ are those for which $t\le 5$ since $72=2\cdot 6^2+0$ is exceptional. There are $14$ such values.

Suppose that $N\in A_t \cup B_{t+1}$ and $N\ge 60$, the smallest member of $A_5$. Then,  $t\ge 5$ and $2t^2+2t\le N$ so $2t<\sqrt{2n}.$ If we show that there are fewer than $16t$ exceptional $n<N$, then the probability that a randomly chosen $ n \le N$ is exceptional is less than $\frac{16t}{N}<\frac{8\sqrt{2N}}{N}=\sqrt{\frac{128}{N}}$.  In each pair $A_j,B_{j+1}$ for $5 \le t$, there are $16$ exceptional values $n$, so less than $60+16(t-4)$ exceptional $n<N.$

\end{proof}

\begin{proposition} \label{prop:shiftedints}

The exceptional $n$ are exactly the integers of the following $16$ standard forms.

\begin{enumerate}
\item The values $n=2t^2+2t+c$ for $c\in\{0,2\}$ with $c \le 2t+1$.
\item The values $n=2t^2+c$  for $c\in\{1,3,2t-1\}$ with $c \le 2t-1.$
\item The values $n=2t^2+2t+c$ for $c\in\{1,3,5,2t-1,2t+1\}$ with $ c \le 2t+1.$
\item The values $n=2t^2+c$ for $c\in\{0,2,4,6,2t-4,2t-2\}$ with $c \le 2t-1.$
\end{enumerate}

Every exceptional $n$ also has at least one of the following $16$ alternate forms.
\begin{enumerate}
\item The values $n=2t^2+2t+c$ for $c\in\{-4,-2,0,2\}$.
\item The  values $n=2t^2+c$  for $c\in\{-3,-1,1,3\}$.
\item The values $n=2t^2+2t+c$ for $c\in\{-1,1,3,5\}$.
\item The values $n=2t^2+c$ for $c\in\{0,2,4,6\}$.
\end{enumerate}
\end{proposition}

\begin{proof}

  The exceptional values are exactly those not enumerated by Proposition \ref{prop4cyclesINnMostGn}. The nonnegative integers are partitioned into intervals $A_t$ and $B_t.$ The standard forms use this same partition and each exceptional value is of only one of those forms. The alternate forms  replace five of the standard forms with forms having $c<0.$ This is valid because $2t^2+2t+(2t-1)=2(t+1)^2+(-3)$, $2t^2+2t+(2t+1)=2(t+1)^2+(-1)$, and  $2t^2+(2t-j)=(2t^2+2t)+(-j)$ for $j \in \{1,2,4\}.$ Because the upper bounds were discarded, the alternate form is not always unique when $t<6.$  It is not hard to check that every nonnegative integer of these alternate forms is exceptional, but that fact is not needed.
\end{proof}

For each of the $16$ alternate forms, we wish to show that there is a $4$-cycle in $G_n$ for all $n>73$ of that form. One of those cases is  $n=2t^2+2.$ There we are considering $n=2,10,20,34,52,74,\ldots .$ Our method is to  give a particular $4$-cycle for  $n_0=34$ and show that it begins a one-parameter family of $4$-cycles for all $n=2t^2+2 \ge n_0$. The same method is used for the other $15$ alternate forms. In each case, we use the smallest $n_0$ possible. For $n=2t^2+1$, there are no $4$-cycles in $G_n$ for $n \in \{3,9,19,33,51,73\}$ and we will start with $n_0=99.$ In the  $15$ other cases, we can choose $n_0<73.$

\begin{example} \label{ex:2tp2b} \normalfont

In the case $n=2t^2+2$ we list particular $4$-cycles in $G_n$ for $n=34,\,52$,  and $74$ corresponding to $t=4,\,5$, and $6$.
Spaces have been added for clarity.
These 4-cycles are

\begin{align*}
  2^34^3\phantom{9^212^2}\enspace\underline{8}^2 \rightarrow&
  &2^44^3\phantom{9^212^2\enspace} 7^2 \rightarrow&
&\underline{1}^4_33^3\underline{6}^2\phantom{10^212^2\enspace}{9} \rightarrow& &
1\,3^35^3\enspace\phantom{10^212^2} {9}\rightarrow\ ,\phantom{\text{ and}} \\
    2^34^37^2\phantom{\enspace 9^2}\underline{10}^2 \rightarrow&
    &2^44^37^2\phantom{11^2\enspace} 9^2\rightarrow&
&\underline{1}^4_33^3\underline{6}^28^2\phantom{10^2\enspace }11 \rightarrow& &
    1\,3^35^38^2\enspace\phantom{10^2} 11 \rightarrow\ , \text{ and} \\
  2^34^37^29^2\enspace \underline{12}^2 \rightarrow&  &2^44^37^29^2\enspace \,11^2 \rightarrow&   &\underline{1}^4_33^3\underline{6}^28^210^2\enspace{13} \rightarrow& & 1\,3^35^38^210^2\enspace 13 \rightarrow\ .
  \phantom{\text{ and}}
\end{align*}

\end{example} 

It seems clear that there is a pattern that continues and gives a $4$-cycle for each  $n=2t^2+2$ for $t \ge 4$. The next three theorems allow for a one-sentence proof, given in case \ref{prop4cyclePlus2a_S} of Proposition \ref{propfamilies}. That proof simply notes that the first line is a valid cycle $C_0$ in $G_{34}$ and that each arc lifts according to certain rules giving a cycle $C_1$ that again lifts.

It is possible to give a proof similar to Proposition \ref{prop4cycle3parameter1} using a scheme with first member $2^34^3 (\overbrace{7\ldots 2y+5}^y)(\underline{2y+6})^2$ for $y \ge 0.$ However, there are $16$ cases to explain, and the method we use instead applies to other cycles of length $4$, and more. For example, the unique $8$-cycles of $G_6$ and $G_{14}$ belong to a family of $8$-cycles in $G_n$ for $n=2t^2+2t+6.$

We first establish some notation. Example \ref{ex:pat} illustrates the notation using the final arcs of the cycles in Example \ref{ex:2tp2b}.

\begin{definition}
A {\bf variable part} is an expression of the form $(m+c)^i$ or $(m+j+c)^{i}$, where $m$ and $j$ are variables taking positive integer values, and $c, i$ are integer constants. A {\bf pattern} is an expression $H s_1 \rightarrow H^* s_2$ where $H$ and $H^*$ are variables taking integer partitions as values, while $s_1$ and $s_2$ are strings whose members are variable parts using the same variables $m$ and $j$. The constant $c$ and exponent $i$ can vary from part to part. Consider an arc $\lambda_1\rightarrow \lambda_2$, where $\lambda_1$ and $\lambda_2$ are integer partitions. We say that the arc {\bf matches}  a particular pattern if there is an assignment of integer partitions to $H$ and $H^*$ and integer values to $m$ and $j$ that transform the pattern into the arc. We further require that the assignments are such that no part of $H$ is as large as a part of $s_1$ and that no part of  $H^{*}$ is as large as a part of $s_2.$
\end{definition}

\begin{example} \normalfont

 \label{ex:pat} We will say an $\Omega$-arc $\lambda_1\rightarrow \lambda_2$ matches the pattern  \newline
$H\quad(m+j) \rightarrow H^*\quad (m+j-1)^2$ if we can assign positive integer values to $m$ and $j$ so that the largest parts of $\lambda_1$ and $\lambda_2$ have sizes $m+j$ and $m+j-1$, with $m+j$ occurring once and $m+j-1$ occurring twice. Then, $H$ and $H^*$ are assigned the partitions consisting of all the smaller parts of $\lambda_1$ and $\lambda_2$. It follows that $\lambda_1$ has $m+j-1$ parts, since the second $(m+j-1)$-part must be the new part of $\lambda_2.$ This means that $H^*$ is all the parts of $H$ reduced by $1.$

Consider the last arcs of the cycles in Example \ref{ex:2tp2b}:
\begin{align*}
   &(e)\text{\hspace{1.5in}} &1\,3^35^3\enspace\phantom{8^210^2x1} 9 \qquad  \rightarrow& &2^34^3\phantom{7^29^2xx}&\ {8}^2\ , \phantom{\text{ \ \ and}} \text{\hspace{1.5in}}  \\
     &(e')\text{\hspace{1.5in}} &1\,3^35^38^2\phantom{10^2}\quad 11 \qquad \rightarrow& &2^34^37^2\phantom{9^2}\quad &{10}^2\ , \text{ \ \ and}
     \text{\hspace{1.5in}}\\
     &(e'')\text{\hspace{1.5in}} &1\,3^35^38^210^2\quad 13\qquad \rightarrow& &2^34^37^29^2\quad &{12}^2\ .   \phantom{\text{ \ \ and}}  \text{\hspace{1.5in}}
\end{align*}

These three $\Omega$-arcs  match the pattern $$H\quad(m+j) \rightarrow H^*\quad (m+j-1)^2.$$
For the arc $(e)$  with $(m,j)=(7,2),\ H= 1\,3^35^3$, and $H^*=2^34^3.$\newline
For the arc $(e')$ with  $(m,j)=(9,2)$, $H= 1\,3^35^38^2$, and $H^*=2^34^37^2.$ \newline
For the arc $(e'')$ with $(m,j)=(11,2),$  $H= 1\,3^35^38^210^2$ and $H^{*}=2^34^37^29^2.$\newline
\end{example}

We will show in Theorems \ref{omegalifts} and \ref{alphalifts} that (as in Example \ref{ex:pat}) for certain patterns, if an arc $e \in G_{n_0}$ matches that pattern with $m=m_0$, then there is an arc $e'\in G_{n_0+2m_0}$ which matches the same or a similar pattern with $m=m_0+2$. Then, we will show in Theorem \ref{1paramfamilies} that, under certain circumstances concerning the arcs of a $2j$-cycle $C_0\in G_{n_0}$ and a fixed value $m=m_0$, there is a cycle $C_1 \in G_{n_0+2m_0}$ which leads to a family of $2j$-cycles, as in Example \ref{ex:2tp2b}.

The function max gives the largest part of a partition, e.g. $\max(1 3^3 5^3) = 5.$
\begin{example}\label{ex:Om4} \normalfont Suppose that $G_n$ has an $\Omega$-arc $e$  matching this pattern:  $$\lambda=H\quad(m+j) \rightarrow H^*\quad (m+j-1)^2$$ for    $j \ge 1$, $m=m_0$, and such that $\max(H)\le m-1.$  It follows that $G_{n'}=G_{n+2m+4}$ has an $\Omega$-arc $e':$ $$\lambda'=H(m+1)^2\quad(m+j+2) \rightarrow H^* m^2\quad (m+j+1)^2.$$ To see this, note that $\lambda$ must have  $m+j-1$ parts so that the new part in $\Omega(\lambda)$ is an (additional) $(m+j-1)$-part. Then the partition $\lambda'=H\,(m+1)^2\quad (m+j+2)$ is separated, since $\max(H)\le m-1,$ and is a node in $G_{n'}$ for $n'=n+2(m+1)+2=n+2m+4.$ Here, $\lambda'$ has $2$ more parts than $\lambda$, so $m+j+1$ parts. Then the $\Omega$-move reduces the parts of $$\lambda'=H\,(m+1)^2\quad(m+j+2)$$ by $1$ and the new part is an (additional) $(m+j+1)$-part yielding $$H^*m^2\quad (m+j+1)^2.$$

 We say that arc $e$ {\bf lifts} to arc $e'$.
 Note that $e'$ has the same pattern for $m=m_0+2$
 that $e$ has for $m=m_0.$  \end{example}

The four arcs of the first cycle, call it $C_0$, in Example \ref{ex:2tp2b} match the top patterns  for items $\alpha_4, \Omega_2, \alpha_5, \Omega_4$ of Theorems \ref{omegalifts} and \ref{alphalifts}, all for $m=7$. The next two cycles are the same for $m=9$ and $m=11.$ As we will see, in Theorem \ref{1paramfamilies}, once one verifies that each arc lifts and that the four new arcs have the same patterns, we have established that there is a one-parameter family of $4$-cycles for $n=2t^2+2$ with $n \ge 32.$ This is case \ref{prop4cyclePlus2a_S} of Proposition \ref{propfamilies}, which simply notes that  $C_0$ is a valid $4$-cycle and that each arc lifts with $m=7$ giving a cycle $C_1$ that again lifts.

The proof of the next theorem  is illustrated in greater detail, for the particular case $(\Omega_4)$, in Example \ref{ex:Om4}.

\begin{theorem} \label{omegalifts}  For each of the five pairs of patterns below, if $G_n$  has an $\Omega$-arc $e$ that matches the first pattern, then $G_{n+2m+4}$ has an $\Omega$-arc $e'$ that matches the second pattern.
\begin{itemize}
\item [$(\Omega_1)$]
\hspace{0.1in}Suppose $\max(H)\le m-3$.
Then,
\begin{align*}
   \ \phantom{m} H\,\phantom{(-1)^2} (m-1)^2
   \rightarrow& \
   \  H^*\,\phantom{(m-2)^2} (m-2)^2(m+2)
   \  \text{ lifts to} \\
     H\,(m-1)^2\  (m+2)^2
     \rightarrow& \
     H^*\,(m-2)^2\  (m+1)^2(m+4).
\end{align*}

\item[$(\Omega_2)$]
\hspace{0.1in}Suppose $\max(H)\le m-2$.
Then,
\begin{align*}
    \phantom{\ }H\,\phantom{m^2}  m^2 \quad \;
    \rightarrow& \
    \phantom{\  -1} H^*\,\phantom{(m1)^2} (m-1)^2(m+j)
    \quad \  \text{ lifts to} \\
   H\,m^2 \  (m+2)^2
   \rightarrow& \
   H^*\,(m-1)^2 \   (m+1)^2(m+j+2).
\end{align*}

\item[$(\Omega_3)$]
\hspace{0.1in}Suppose $\max(H)\le m-2$.
Then,
\begin{align*}
    \  H\,\phantom{m^2}  (m+1)^2
    \rightarrow& \
    \phantom{(m)^2} H^*\phantom{-1}\ \,  m^2(m+j)
   \quad\quad \  \text{ lifts to} \\
   H\,m^2\  (m+3)^2
   \rightarrow& \
   H^*\,(m-1)^2\  (m+2)^2(m+j+2).
\end{align*}

\item[$(\Omega_4)$]
\hspace{0.1in}Suppose $\max(H)\le m-1$ and $j\ge 1$.
Then,
\begin{align*}
    \  H\,\phantom{(m+1)^2}   (m+j)
    \  \;\; \rightarrow& \
    \  H^*\,\phantom{m^2}  (m+j-1)^2
    \   \text{ lifts to} \\
   H\,(m+1)^2\  (m+j+2)
   \rightarrow& \
   H^*\,m^2\  (m+j+1)^2.
\end{align*}

\item[$(\Omega_5)$]
\hspace{0.1in} Suppose $\max(H)\le m-3$ and $j \ge 2$.
Then,
{\small
\begin{align*}
    \  \phantom{m}H\,\phantom{(-1)^2} \  \quad m^2(m+j) \quad
    \rightarrow& \
    \  \phantom{m} H^*\,\phantom{(-1)^2} \  (m-1)^2(m+j-1)^2
    \  \text{ lifts to} \\
   H\,(m-1)^2 \  (m+2)^2(m+j+2)
   \rightarrow& \
   H^*\,(m-2)^2 \  (m+1)^2(m+j+1)^2.
\end{align*}
    }
\end{itemize}

\end{theorem}

Note that for a given $m$, no $\Omega$-arc $\lambda_1 \rightarrow \Omega(\lambda_1)$ can match more than one of the patterns.

\begin{proof}
    Let the first arc be $\lambda_1\rightarrow \lambda_2$ and the second arc be $\lambda_1'\rightarrow \lambda_2'.$ Let $p$ be the largest part of  $\lambda_2$. Then, $p$ is the new part of $\lambda_2$ and $\lambda_1$ has $p$ parts. It is easy to see that $\lambda_1'$ is in $G_{n+2m+4}$ and has two more parts than $\lambda_1$ so the new part of $\lambda_2$ is a $p+2$-part. The condition on $\max(H)$  means that $\lambda_1'$ is separated. Furthermore, $\lambda_2'$ results from $\lambda_1'$ by the $\Omega$-move since reducing the parts of $H$ by one gives $H^*$, all the other parts reduce by $1$, and then the new part is a $p+2$-part.
\end{proof}

Instead of the five pairs of patterns above, we will need nine pairs of patterns for $\alpha$-moves. This is because an arc matching $H\,\underline{m}^2 \rightarrow H^*\,(m-1)^2$ is a different case than an arc matching $H\,m^2 \rightarrow H^*\,m^2.$ The parts changed in $H$ can be determined by comparison to $H^*$. For clarity, however, a partition assigned to $H$ can be thought to have underlines using Notation \ref{note_alphamove}.

\begin{theorem} \label{alphalifts} For each of the nine pairs of patterns below, if $G_n$  has an $\alpha$-arc that matches the first pattern, then $G_{n+2m+4}$ has an $\alpha$-arc that matches the second pattern for the same $H$ and $H^*.$

\begin{itemize}
\item[$(\alpha_1)$]
\hspace{0.1in}Suppose $\max(H)\le m-2$. Then,
\begin{align*}
    H\, \phantom{\underline{m}^2}
     \underline{m}^2\phantom{m+2}
    \rightarrow&\quad
     H^*\, \phantom{(m-1)^2}( m-1)^2\phantom{(m)}
    \text{ lifts to}\\
    H\,\underline{m}^2\  (m+2)^2\rightarrow& \
    H^*\,(m-1)^2 \ (m+2)^2.
     \phantom{\text{ lifts to}\ }
\end{align*}

\item[$(\alpha_2)$]
\hspace{0.1in}Suppose $\max(H)\le m-3$ and $j\ge 2$. Then,
{\small
\begin{align*}
    \,\, H \phantom{(m-1)^2\ }
    \underline{m}^2(\underline{m+j})
    \phantom{(m+j)}
     \rightarrow& \quad
    H^* \phantom{m-1}\,\
    (m-1)^2(m+j-1)   \text{ lifts to} \\
    H\,(m-1)^2\  (\underline{m+2})^2(\underline{m+j+2})
    \rightarrow& \
    H^*\,(m-1)^2\  (m+1)^2(m+j+1).
\end{align*}
}
\item[$(\alpha_3)$]
\hspace{0.1in}Suppose $\max(H)\le m-1$ and $j\ge 1$. Then,
\begin{align*}
    \  H\, \phantom{(m+1)}\,\  (\underline{m+j})\
    \rightarrow& \quad
    H^*\,\phantom{(m+1)^2}(m+j-1)
    \ \text{ lifts to} \\
    H\,(m+1)^2 \  (\underline{m+j+2})\rightarrow& \
    H^*\,(m+1)^2\  (m+j+1).
\end{align*}

\item[$(\alpha_4)$]
\hspace{0.1in}Suppose $\max(H)\le m-2$ and $j\ge 1$. Then,
\begin{align*}
   \  H\,\phantom{m^2}\  (\underline{m+j})^2\phantom{m}
    \rightarrow& \quad
    H^*\,\phantom{m^2} (m+j-1)^2
   \  \text{ lifts to}\\
    H\,m^2\  (\underline{m+j+2})^2
    \rightarrow& \
    H^*\,m^2\  (m+j+1)^2.
\end{align*}

\item[$(\alpha_5)$]
\hspace{0.1in}Suppose  $\max(H)\le m-1$ and $j\ge 1$.
Then,
\begin{align*}
   \  H\,\phantom{(m+1)^2}(m+j)\;\;
    \rightarrow& \quad
    H^*\,\phantom{(m+1)^2}\  (m+j)
    \  \text{ lifts to}\\
    H\,(m+1)^2\  ({m+j+2})
    \rightarrow& \
    H^*\,(m+1)^2\  (m+j+2).
\end{align*}

\item[$(\alpha_6)$]
\hspace{0.1in}Suppose $\max(H)\le m-3$ and $j\ge 2$.
Then,
\begin{align*}
   \quad H\,\phantom{(m-1)^2}\  m^2(m+j)\quad
    \rightarrow& \quad \
    H^*\, \phantom{(m-1)^2}\  m^2(m+j)
    \  \text{ lifts to}\\
    H\,(m-1)^2\  (m+2)^2({m+j+2})
    \rightarrow& \
    H^*\,(m-1)^2\ (m+2)^2({m+j+2}).
\end{align*}

\item[$(\alpha_7)$]
\hspace{0.1in}Suppose $\max(H)\le m-3$ and $j\ge 2$.
Then,
\begin{align*}
   \quad H\,\phantom{(m-1)^2}\  \underline{m}^2(m+j)
    \quad
    \rightarrow& \quad \
    H^*\,\phantom{(m-1)^2} (m-1)^2 (m+j)
    \  \text{ lifts to}\\
    H\,(m-1)^2\  (\underline{m+2})^2({m+j+2})
    \rightarrow& \
    H^*\,(m-1)^2\  (m+1)^2({m+j+2}).
\end{align*}

\item[$(\alpha_8)$]
\hspace{0.1in}Suppose $\max(H)\le m-2$ and $j\ge 0$.
Then,
\begin{align*}
   \  H\, \phantom{m^2} \  ({m+j})^2 \
    \rightarrow& \quad
    H^*\, \quad (m+j)^2\phantom{(m23^2}
     \  \text{ lifts to}\\
    H\,m^2\  ({m+j+2})^2
    \rightarrow& \
    H^*\,m^2\  (m+j+2)^2.
\end{align*}
\item[$(\alpha_9)$]
\hspace{0.1in}Suppose $\max(H)\le m-4$ and $j\ge 1$.
Then,
\begin{align*}
   \quad H\,\phantom{(m-2)} (\underline{m-1})^2(m+j)^2
    \;\;\;\;
    \rightarrow& \
    \ \;\; H^*\,\phantom{(m)} \;\;\;\; (m-2)^2
    (m+j)^2
     \  \text{ lifts to}\\
    H\,(m-2)^2\ (\underline{m+1})^2({m+j+2})^2
    \rightarrow& \
    H^*\,(m-2)^2\  m^2({m+j+2})^2.
\end{align*}

\end{itemize}

\end{theorem}
Note that if $m$ is given, no $\alpha$-arc $\lambda_1 \rightarrow \lambda_2$ can match more than one of the patterns. Also, $\lambda_2$ is uniquely determined by $\lambda_1$ and the indicated $\alpha$-move.
\begin{proof}
 Let the first arc be $\lambda_1\rightarrow \lambda_2$ and the second arc be $\lambda_1'\rightarrow \lambda_2'.$ In each case, the new part of $\lambda_2$ is a $p$-part in $H^*$, for the same $p$,  where $p$ is the number of parts changed in the move. It is easy to see that $\lambda_1'$ is in $G_{n+2m+4}$ and has two more parts than $\lambda_1.$ The condition on $\max(H)$ means that $\lambda_1'$ is separated. Furthermore, $\lambda_2'$ results from $\lambda_1'$ by the $\alpha$-move, which changes the same parts in $H$, thus creating a new $p$-part for the same $p.$
\end{proof}

We use the next result $16$ times for our analysis of $4$-cycles. It is no harder to state and prove for cycles of arbitrary even length, and we do so.

\begin{theorem} \label{1paramfamilies} Let  $C_0$ be a $2k$-cycle $$C_0=\lambda_1 \rightarrow \lambda_2 \rightarrow \ldots \rightarrow \lambda_{2k}\rightarrow$$ in $G_{n_0}$. Let $m=m_0$ be a fixed integer and $n_1=n_0+2m_0+4.$ Suppose that $C_0$ satisfies the following conditions:
\begin{itemize} \item Every $\Omega$-arc $\lambda_{i}\rightarrow \lambda_{i+1}$ of $C_0$ matches the first pattern of one of the five cases of Theorem \ref{omegalifts} for $m=m_0$, lifting to the arc  $\lambda'_{i}\rightarrow \lambda'_{i+1}$ in $G_{n_1}.$
\item Every $\alpha$-arc $\lambda_i\rightarrow \lambda_{i+1}$ of $C_0$ matches the first pattern of one of the nine cases of Theorem \ref{alphalifts} for $m=m_0$, lifting to the arc $\lambda'_i\rightarrow \lambda'_{i+1}$ in $G_{n_1}.$
\end{itemize}
Then, for $n_1=n_0+2m_0+4$, there is a $2k$-cycle $$C_1=\lambda_1' \rightarrow \lambda_2' \rightarrow \ldots \rightarrow \lambda_{2k}'\rightarrow $$ in $G_{n_1}$, where each $\lambda_i'$ has two more parts than $\lambda_i$.

Furthermore,
\begin{enumerate}
    \item If $m_0=2j-1$ is odd and $c=n_0-2j^2$, then $n_0=2j^2+c$ and there is a $2k$-cycle in $G_n$ for all $n=2t^2+c$ with $t\ge j.$
    \item If $m_0=2j$ is even and $c=n_0-(2j^2+2j),$ then $n_0=2j^2+2j+c$ and there is a $2k$-cycle in $G_n$ for all $n=2t^2+2t+c$ with $t\ge j.$
\end{enumerate}

\end{theorem}

\begin{proof}  From the two conditions, applying Theorems  \ref{omegalifts} and \ref{alphalifts} to the arcs of $C_0$ gives a sequence of $2k$ arcs in $G_{n_1}$.  We next note that $C_1$ satisfies the two conditions for $m=m_1=m_0+2$, perhaps with different patterns for the arcs. It follows that there is a one-parameter family $C_s$ of $2k$-cycles for $s\ge 0$ with $C_s \in G_{n_s}$ and each arc matches a pattern for $m=m_s$. Here, $m_s=m_0+2s$ and $n_{s+1}=n_s+(2m_s+4).$

Examination shows that in $12$ of the $14$ cases, the (unique) pattern that $\lambda_i \rightarrow \lambda_{i+1}$ matches for $m=m_s$ (including the value of $j$ if applicable) is the same one that the arc it lifts to, $\lambda_i' \rightarrow \lambda_{i+1}'$, matches for $ m=m_s+2$. The two exceptions are as follows and can only occur for $s=0.$
\begin{itemize}
\item If $\lambda_i \rightarrow \lambda_{i+1}$ matches the first pattern of  $\Omega_1$  for $m=m_0$,
then $\lambda_i' \rightarrow \lambda_{i+1}'$ matches the first pattern of $\Omega_2$ for $m=m_1=m_0+2$.
\item If $\lambda_i \rightarrow \lambda_{i+1}$ matches the first pattern of
$\alpha_1$ for $m=m_0$,
then $\lambda_i' \rightarrow \lambda_{i+1}'$ matches the first pattern of
$\alpha_8$ for $m=m_1$  with $j=0$.
\end{itemize}

Since $n_{s+1}=n_s+2m_s+4$ where $m_s=m_0+2s,$ a straightforward induction gives $n_s=n_0+2s(m_0+s+1).$

Suppose first that $m_0=2j-1$ is odd and let $t=s+j.$ Then, $$n_s=n_0+2s(s+2j)=n_0+2(t-j)(t+j)=2t^2+(n_0-2j^2).$$
Similarly, when $m_0=2j$ and $t=s+j$, \newline  $n_s=n_0+2s(s+2j+1)=n_0+2(t-j)(t+j+1)=2t^2+2t+(n_0-(2j^2+2j)).$
\end{proof}

\begin{proposition} \label{propfamilies2t}
There is a one-parameter family of $4$-cycles in each of the following cases:
\begin{enumerate}

\item   \label{prop4cycle2TMinus4_S} $n=2t^2+2t-4$ for $n\ge 56$,
\item   \label{prop4cycle2TMinus2_S}  $n=2t^2+2t-2$ for $n\ge 38$,
\item   \label{prop4cycle2TMinus1_S}  $n=2t^2+2t-1$ for  $n\ge 59$,
\item  \label{prop4cycle2tPlus0_S} $n=2t^2+2t$ for  $n\ge 24$,
\item \label{prop4cycle2tPlus1_S} $n=2t^2+2t+1$ for  $n\ge 61$,
\item \label{prop4cycle2tPlus2_S}  $n=2t^2+2t+2$ for  $n\ge 26$,
\item \label{prop4cycle2tPlus3_S}  $n=2t^2+2t+3$ for  $n\ge 63$, and
\item  \label{prop4cycle2tPlus5_S}  $n=2t^2+2t+5$ for  $n\ge 45$.
 \end{enumerate}

 These include the following $14$ values listed in order of the cases above \newline $n=56,\ 38,58,\ 59,\ 24,40,60, \ 61,\ 26,42,62,\  63,\ 45,$ and $65.$

 \end{proposition}

 \begin{proof}
     In each case, the sequence of arcs is a $4$-cycle $C_0$ in $G_{n_0}$ that satisfies the two conditions of Theorem \ref{1paramfamilies}. This is because the arcs (in the order given) match the indicated patterns from Theorems \ref{omegalifts} and \ref{alphalifts} for  $m=m_0$. Furthermore,  examination shows that the corresponding endpoints of the lifted arcs agree. Thus, $C_0$ lifts to a cycle $C_1$ and is the first member of a one-parameter family  providing cycles for all $n \ge n_0$ of the appropriate form $2t^2+2t+c.$ In seven of the cases, the subsequent cycles $C_1,C_2,\ldots$ of the one-parameter family have arcs that match the same patterns as $C_0.$ In case $3$, $\alpha_8$ and $\Omega_2$ replace $\alpha_1$ and $\Omega_1$ as noted in the proof of Theorem \ref{1paramfamilies}.

\begin{enumerate}

\item For the case $n_0=56$ and $m_0=10$, we have the cycle \newline
$2^4\underline{6}^3\underline{9}^212
\rightarrow 2^45^48^212
\rightarrow \underline{1}^4_1\underline{4}^47^2\underline{11}^2
\rightarrow 1^33^47^310^2  \rightarrow$  \newline
 using patterns from $\alpha_5$, $\Omega_4$, $\alpha_4,$ and $\Omega_2.$

 \item For the case $n_0=38$ and $m_0=8$, we have the cycle\newline
$2^34^2\underline{7}^210   \rightarrow 2^44^26^210   \rightarrow
\underline{1}^4_1{3}^25^2\underline{9}^2  \rightarrow 1^3 3^3 5^2 8^2 \rightarrow$ \newline
using patterns from $\alpha_5$, $\Omega_4$, $\alpha_4,$ and $\Omega_2.$

\item For the case $n_0=59$ and $m_0=10$, we have the cycle \newline
$2^34^37^3\underline{10}^2 \rightarrow 2^44^37^39^2
\rightarrow \underline{1}^4_43^3\underline{6}^38^2\underline{12}
\rightarrow 3^35^38^311 \rightarrow$ \newline
 using patterns from $\alpha_1$, $\Omega_1$, $\alpha_3,$
 and $\Omega_4.$\newline
Subsequent cycles of the resulting one-parameter family use patterns from  $\alpha_8$, $\Omega_2$, $\alpha_3,$ and $\Omega_4.$

\item For the case $n_0=24$ and $m_0=6$, we have the cycle \newline
$2^3\underline{5}^28  \rightarrow 2^44^28
\rightarrow  \underline{1}^4_13^2\underline{7}^2
\rightarrow 1^33^36^2  \rightarrow$ \newline
 using patterns from $\alpha_5$, $\Omega_4$, $\alpha_4,$ and $\Omega_2.$

\item For the case $n_0=61$ and $m_0=10$, we have the cycle \newline
$2^34^37^3\underline{11}^2 \rightarrow
2^44^37^310^2 \rightarrow
\underline{1}^4_33^3\underline{6}^3\underline{9}^212
\rightarrow 1\,3^35^38^312 \rightarrow$ \newline
using patterns from $\alpha_4$, $\Omega_2$, $\alpha_5$, and $\Omega_4$.

\item For the case $n_0=26$ and $m_0=6$, we have the cycle\newline
$2^3\underline{6}^28   \rightarrow 2^45^28   \rightarrow
\underline{1}^4_1\underline{4}^27^2 \rightarrow
1^33^37^2 \rightarrow$ \newline
using patterns from $\alpha_7$, $\Omega_4$, $\alpha_8,$ and $\Omega_3.$

\item For the case $n_0=63$ and $m_0=10$, we have the cycle \newline
$2^44^2\underline{7}^2\underline{10}^213 \rightarrow
2^44^36^29^213  \rightarrow \underline{1}^4_13^35^28^2\underline{12}^2
\rightarrow 1^33^45^28^211^2  \rightarrow$ \newline
using patterns from $\alpha_7$, $\Omega_4$, $\alpha_4,$ and $\Omega_3.$

\item For the case $n_0=45$ and $m_0=8$, we have the cycle \newline
$2^4\underline{5}^2\underline{8}^211  \rightarrow 2^44^37^211
\rightarrow \underline{1}^4_13^36^2\underline{10}^2
\rightarrow 1^33^46^29^2  \rightarrow$ \newline
using patterns from $\alpha_7$, $\Omega_4$, $\alpha_4,$ and $\Omega_3.$
\end{enumerate}
\end{proof}

\begin{proposition} \label{propfamilies}
There is a one-parameter family of $4$-cycles in each of the following cases:

\begin{enumerate}
\item \label{prop4cycleMinus3_S}  $n=2t^2-3$ for $n \ge 47$,
\item \label{prop2TSquaredMinus1} $n=2t^2-1$ for $n \ge 71$,
\item  \label{propValue2TSquared_S} $n=2t^2$ for $n \ge 32$,
\item \label{prop2TSquaredPlus1} $n=2t^2+1$ for $n \ge 99$,
\item  \label{prop4cyclePlus2a_S}  $n=2t^2+2$ for $n \ge 34$,
\item  \label{prop4cyclePlus3_S} $n=2t^2+3$ for $n \ge 35$,
\item   \label{prop4cyclePlus4_S} $n=2t^2+4$ for $n \ge 54$, and
\item   \label{prop4cyclePlus6_S} $n=2t^2+6$ for $n \ge 24$.
\end{enumerate}

These include the following $14$ values of $n<73$ and listed in order of the cases above:
 $n=47, 69,\ 71,\ 32, 50, 72,\ 34, 52,\ 35, 53, 54,\ 24, 38$, and $56$.

\end{proposition}
\begin{proof}
In each case, the sequence of arcs is a $4$-cycle $C_0$ in $G_{n_0}$ and satisfies the two conditions of Theorem \ref{1paramfamilies}. This is because the arcs (in the order given) match the indicated patterns from Theorems \ref{omegalifts} and \ref{alphalifts} for  $m=m_0$ and examination shows that the corresponding endpoints of the lifted arcs agree.  Thus, $C_0$ lifts to a cycle $C_1$ and is the first member of a one-parameter family  providing cycles for all $n \ge n_0$ of the appropriate form $2t^2+c.$  In seven of the cases, the subsequent cycles $C_1,C_2,\ldots$ of the one-parameter family have arcs that match the same patterns as $C_0.$ In case $3$, $\alpha_8$ and $\Omega_2$ replace $\alpha_1$ and $\Omega_1$ as noted in the proof of Theorem \ref{1paramfamilies}.
 \begin{enumerate}

 \item  For the case $n_0=47$ and $m_0=9$, we have the cycle \newline
$2^{4} \underline{5}^{2} \underline{9}^{2} 11 \rightarrow 2^{4} 4^{3} 8^{2} 11 \rightarrow \underline{1}_1^{4} 3^{3} \underline{7}^{2} 10^{2} \rightarrow 1^{3} 3^{4} 6^{2} 10^{2}\rightarrow$\newline
using patterns from $\alpha_7$, $\Omega_4$, $\alpha_8$, and $\Omega_3$.

 \item For the case $n_0=71$ and $m_0=11$, we have the cycle \newline
$2^{4} {5}^{3} \underline{8}^{3} \underline{12}^{2} \rightarrow  2^{4} 5^{4} 7^{3} 11^{2}  \rightarrow
  \underline{1}_3^{4} \underline{4}^{4} {6}^{3} \underline{10}^{2} {13}  \rightarrow  1\, 3^{4} 6^{3}\,\, 9^{3} 13 \rightarrow $\newline
using patterns from $\alpha_4$, $\Omega_2$, $\alpha_5$, and $\Omega_4$.

 \item For the case $n_0=32$ and $m_0=7$, we have the cycle \newline
 $ 2^34^3\underline{7}^2
 \rightarrow 2^44^36^2
 \rightarrow \underline{1}^4_43^35^2\underline{9}
 \rightarrow 3^35^38  \rightarrow$ \newline
 using patterns from $\alpha_1$, $\Omega_1$, $\alpha_3,$ and $\Omega_4.$\newline
 Subsequent cycles of the resulting one-parameter family use patterns from  $\alpha_8$, $\Omega_2$, $\alpha_3,$ and $\Omega_4.$

 \item For the case $n_0=99$ and $m_0=13$, we have the cycle \newline
$2^{4} {5}^{3} \underline{8}^{3} \underline{12}^{2} 14^{2} \rightarrow  2^{4} 5^{4} 7^{3} 11^{2} 14^{2}  \rightarrow  \underline{1}_3^{4} \underline{4}^{4} {6}^{3} \underline{10}^{2} 13^{2} 15 \rightarrow  1\, 3^{4} 6^{3} 9^{3} 13^{2}15 \rightarrow $\newline
using patterns from $\alpha_9$, $\Omega_3$, $\alpha_6$, and $\Omega_5$.

 \item For the case $n_0=34$ and $m_0=7$, we have the cycle \newline
 $2^34^3\underline{8}^2 \rightarrow 2^44^37^2
 \rightarrow \underline{1}^4_33^3\underline{6}^2{9}
 \rightarrow 1\,3^35^3{9}  \rightarrow$ \newline
 using patterns from $\alpha_4$, $\Omega_2$, $\alpha_5,$ and $\Omega_4.$

 \item For the case $n_0=35$ and $m_0=7$, we have the cycle \newline
 $\underline{3}^4\underline{7}^29
 \rightarrow 2^46^39
 \rightarrow \underline{1}^4_1\,\underline{5}^38^2
 \rightarrow 1^34^48^2 \rightarrow$ \newline
 using patterns from $\alpha_7$, $\Omega_4$, $\alpha_8,$ and $\Omega_3.$

 \item For the case $n_0=54$ and $m_0=9$, we have the cycle \newline
 $3^35^3\underline{9}^2\underline{12}
 \rightarrow 3^45^38^211
 \rightarrow  \underline{2}^44^3\underline{7}^210^2
 \rightarrow  1^44^36^310^2 \rightarrow $ \newline
 using patterns from $\alpha_2$, $\Omega_4$, $\alpha_8,$ and $\Omega_3.$

 \item For the case $n_0=24$ and $m_0=5$, we have the cycle \newline
 $2^3\underline{5}^28  \rightarrow 2^44^28
 \rightarrow \underline{1}^4_13^2\underline{7}^2
 \rightarrow 1^33^36^2  \rightarrow$ \newline
 using patterns from $\alpha_7$, $\Omega_4$, $\alpha_4,$ and $\Omega_3.$
  \end{enumerate}

\end{proof}

\begin{theorem}\label{thm4cyclesInGn}
    There is a $4$-cycle in $G_n$  for $n>73$.
\end{theorem}

\begin{proof}
    If $n$ is not exceptional, then, from Proposition  \ref{prop4cyclesINnMostGn}, there is a $4$-cycle in $G_n$ even if $n<73.$  From Proposition \ref{prop:shiftedints}, it follows that every exceptional $n$  has (at least) one of the $16$ forms  $2t^2+c$ for $c \in \{-3,-1,0,1,2,3,4,6\}$ and $2t^2+2t+c$ for $ c \in \{-4,-2,-1,0,1,2,3,5\}.$ The $16$ items of Propositions \ref{propfamilies2t} and \ref{propfamilies} provide $4$-cycles for all $n>73$ of these $16$ forms.
\end{proof}

Together with this result, the following theorem shows that there is a $4$-cycle in $G_n$ for all $n \ge 16$ with $19$ exceptions, the largest of which is $73$.

\begin{theorem}\label{thm4cyclesInGn_more}
    There is a $4$-cycle in $G_n$  for the following $39$ values of $n<73.$
    \begin{itemize}
        \item The $25$ even values $16 \le n \le 72$ other than $18,20,22,$ and $36.$
        \item The $4$ values $n= 35,37,45,$ and $47$, together with the $10$ odd values $53 \le n \le 71.$
    \end{itemize}
      These are exactly the nonnegative $n\le 73$ except for the $35$ values listed in Remark \ref{re:no4cycs} as having no $4$-cycles. Among these $35$ values are the values $n \in \{0,1,\ldots,15\}$. It follows that there are $4$-cycles in $G_n$ for all $n \ge 16$ with $19$ exceptions.
\end{theorem}

\begin{proof}
Proposition \ref{prop4cyclesINnMostGn} lists $14$ values of $n<73$ such that $G_n$ has a $4$-cycle $Q(x,y,z)$. Propositions \ref{propfamilies2t} and \ref{propfamilies} each listed $14$ exceptional $n<73$ such that $G_n$ has a $4$-cycle. However, there are only $39$ distinct values because $26,38,$ and $56$ are  listed twice each. Why this happens is interesting and explained below. Examination shows that these are the $39$ values specified in this theorem. Inspection also shows that these are distinct from the $35$ values of Remark \ref{re:no4cycs}. Those are the members of $\{0,1,\cdots,15\}$, along with $19$ values $n>16.$
\end{proof}

The repeated values $26,38,$ and $56$ occur for the following reasons. There is a $4$-cycle for $n_0=24$ that lifts with $m_0=6$ to give a family for $2t^2+6$ starting at $t=3$. This family gives $4$-cycles for $n=38$ and $56$. The same $4$-cycle for $n_0=24$ also lifts with $m_0=5$ to give a family for $2t^2+2t$. There is a different $4$-cycle for $n_0=38$ that gives a family for $2t^2+2t-2$ and a different $4$-cycle for $n_0=56$ that gives a family for $2t^2+2t-4.$

\section{Directions for further research}

We conclude with a few open questions for the reader to consider.

\begin{itemize}
    \item It seems that, with some small exceptions, the one-part partition $n$ is the unique initial winning position farthest from the recurrent states. Is this true? How far is it?
    \item Prove that there is a losing position for each $n>6.$ The open cases are $n=6m$ and $n=6m+2.$
    \item Find a strategy that is frequently successful in reaching the recurrent states from a random separated partition.
    \item Does each $\Omega$-move in a $4$-cycle create the largest part in the next partition?
    \item Characterize the $4$-cycles. There are over $20$ of them for $n=100.$
    \item We have shown that there is a single one-parameter family of $2$-cycles and that there  are many one-parameter families of $4$-cycles. Can one find one-parameter families of $(2k)$-cycles for every integer $k\ge 3$?

  \end{itemize}

The five lifting rules of Theorem \ref{omegalifts}  and nine lifting rules of Theorem \ref{alphalifts} suffice to treat the $64$ arcs that occur for the $16$ $4$-cycles in the proofs of Propositions \ref{propfamilies2t} and \ref{propfamilies}. These rules have the property that each lift adds two more parts and leaves all but the top one or two part  sizes unchanged. Up to $n=76$, every $4$-cycle lifts for one, two, or three values of $m.$ Consider a directed graph $\mathcal{Q}$ whose nodes represent $4$-cycles, with an arc  $C\rightarrow C'$ if the rules lift $C$ to $C'.$
\begin{itemize}
    \item Are there any $4$-cycles that do not lift? If not, then every component of  $\mathcal{Q}$ is infinite.
    \item What is the structure of $\mathcal{Q}$? Does each component have only one member with in-degree $0$?
    \item  Are there modifications of the lifting rules or additional lifting rules of the same nature that give more edges or merge components in the modified $\mathcal{Q}$?
   \end{itemize}

\section*{Tool and computational resource disclosure}\label{Disclosure}
ChatGPT was used to proofread the manuscript for correct grammar, spelling, and punctuation.

\end{document}